\documentclass[onefignum,onetabnum, onealgnum]{siamart171218}

\usepackage{amsmath,amsfonts,amsopn}
\usepackage{graphicx}
\usepackage[caption=false]{subfig}
\usepackage{epstopdf}
\ifpdf
  \DeclareGraphicsExtensions{.eps,.pdf,.png,.jpg}
\else
  \DeclareGraphicsExtensions{.eps}
\fi
\usepackage{bm}
\usepackage{multirow}

\newsiamremark{remark}{Remark}

\usepackage{algorithm, algorithmic}

\def\intO{\int_{\Omega}}
\def\intOs{\int_{\Omega_s}}
\def\div{\mathrm{div}}
\def\proj{\mathrm{proj}}
\def\Hdiv{H(\div; \Omega)}
\def\Hzdiv{H_0 (\div ; \Omega)}

\def\p{\mathbf{p}}
\def\q{\mathbf{q}}
\def\tp{\tilde{\mathbf{p}}}
\def\Hp{\mathcal{H}_I \p_{\Gamma}}
\def\Hq{\mathcal{H}_I \q_{\Gamma}}

\def\n{\mathbf{n}}
\def\R{\mathbb{R}}
\def\0{\mathbf{0}}

\def\E{\mathcal{E}}
\def\J{\mathcal{J}}
\def\L{\mathcal{L}}
\def\T{\mathcal{T}}
\def\I{\mathcal{I}}
\def\tY{\tilde{Y}}
\def\tC{\tilde{C}}
\def\tp{\tilde{\mathbf{p}}}
\def\tJ{\tilde{\mathcal{J}}}
\def\N{\mathcal{N}}

\def\trid{\mathrm{trid}}

\DeclareMathOperator*{\argmin}{\arg\min}

\usepackage[normalem]{ulem}
\usepackage{color}

\headers{Domain Decomposition for Dual ROF}{C.-O. Lee, E.-H. Park, and J. Park}

\title{A Finite Element Approach for the Dual Rudin--Osher--Fatemi Model and Its Nonoverlapping Domain Decomposition Methods\thanks{Submitted to the editors January 16th, 2018.
\funding{The first author's work was supported by NRF grant funded by MSIT (NRF-2017R1A2B4011627),
the second author's work was supported by the Basic Science Research Program through NRF funded by MOE (NRF-2013R1A6A3A04058166),
and the third author's work was supported by NRF grant funded by the Korean Government (NRF-2015-Global Ph.D. Fellowship Program).}}}

\author{Chang-Ock Lee\thanks{Department of Mathematical Sciences, KAIST, Daejeon 34141, Korea
  (\email{colee@kaist.edu}, \email{jongho.park@kaist.ac.kr}).}
\and Eun-Hee Park\thanks{School of General Studies, Kangwon National University, Samcheok 25913, Korea 
  (\email{eh.park@kangwon.ac.kr}).}
\and Jongho Park\footnotemark[2]}

\ifpdf
\hypersetup{
  pdftitle={A Finite Element Approach for the Dual Rudin--Osher--Fatemi Model and Its Nonoverlapping Domain Decomposition Methods},
  pdfauthor={C.-O. Lee, E.-H. Park, and J. Park}
}
\fi

\begin{document}

\maketitle

\begin{abstract}
We consider a finite element discretization for the dual Rudin--Osher--Fatemi model using a Raviart--Thomas basis for $\Hzdiv$.
Since the proposed discretization has splitting property for the energy functional, which is not satisfied for existing finite difference-based discretizations, it is more adequate for designing domain decomposition methods.
In this paper, a primal domain decomposition method is proposed, which resembles the classical Schur complement method for the second order elliptic problems,
and it achieves $O(1/n^2)$ convergence.
A primal-dual domain decomposition method based on the method of Lagrange multipliers on the subdomain interfaces is also considered.
Local problems of the proposed primal-dual domain decomposition method can be solved at a linear convergence rate.
Numerical results for the proposed methods are provided.
\end{abstract}



\section{Introduction}
Nowadays, due to advances in imaging devices, large scale images have become increasingly available, and there has arisen the necessity of parallel algorithms for image processing.
One suitable method for parallel computation is the domain decomposition method~(DDM), for which we solve a problem by splitting its domain into several smaller subdomains and conquering the small problem in each subdomain separately.
We consider the Rudin--Osher--Fatemi (ROF) model~\cite{ROF:1992} as a model problem,
which is a classical and effective model for image denoising:
\begin{equation}
\label{ROF}
\min_{u \in BV(\Omega)} \frac{\alpha}{2}\intO {(u-f)^2 \,dx} + TV(u),
\end{equation}
where $\Omega$ is the rectangular domain of an image, $f \in L^2 (\Omega)$ is an observed noisy image, $\alpha$ is a positive denoising parameter,
and $TV(u)$ is the total variation measure defined by
\begin{equation*}
\label{TV}
TV(u) = \sup \left\{ \intO {u \div \q \,dx} : \q \in (C_0^1 (\Omega))^2 , |\q | \leq 1 \right\}.
\end{equation*}
Here, $| \q| \leq 1$ means that $|\q (x) | \leq 1$ for a.e.\ $x \in \Omega$.
The solution space $BV(\Omega)$ denotes the space of the functions in $L^1 (\Omega)$ with the finite total variation,
which is a Banach space equipped with the norm $\|u \|_{BV(\Omega)} = \| u \|_{L^1 (\Omega)} + |Du|(\Omega)$.
It is well known that the ROF model has an anisotropic diffusion property so that it preserves edges and discontinuities in images~\cite{SC:2003}.

While overlapping DDMs for image restoration were considered in~\cite{FLS:2010,XTW:2010}, nonoverlapping DDMs for the total variation minimization were proposed in \cite{FS:2009, HL:2013}.
But Lee and Nam~\cite{LN:2017} gave a counterexample that an overlapping DDM does not converge to the global minimizer.
In~\cite{LLWY:2016}, Lee~et al.\ suggested DDMs with the primal-dual stitching technique.
In \cite{CTWY:2015, HL:2015, LN:2017}, DDMs based on the dual total variation minimization were proposed.
In particular, Chang~et al.~\cite{CTWY:2015} showed that the overlapping subspace correction methods for the dual ROF model have $O(1/n)$ convergence.

There are several major difficulties in designing DDMs for~\cref{ROF}.
First, the energy functional in~\cref{ROF} is nonsmooth, which makes the design of solvers hard.
In addition, the energy functional is nonseparable in the sense that it cannot be expressed as the sum of the local energy functionals in the subdomains due to the total variation term. 
Finally, the solution space $BV(\Omega)$ allows discontinuities of a solution on the subdomain interfaces, so that it is difficult to design an appropriate interface condition of a solution.
One way to overcome such difficulties is to consider the Fenchel--Rockafellar dual problem as in \cite{CTWY:2015, HL:2015, LN:2017},
which is stated as
\begin{equation}
\label{dual_ROF_old}
\min_{\p \in (C_0^1 (\Omega))^2} \frac{1}{2\alpha} \intO ( \div \p + \alpha f )^2 \,dx \hspace{0.5cm}
\textrm{subject to } |\p| \leq 1.
\end{equation}
Even if it is cumbersome to treat the inequality constraint $|\p| \leq 1$, \cref{dual_ROF_old} is more suitable for DDMs,
since the energy functional is separable and the solution space $(C_0^1 (\Omega))^2$ has some regularity on the subdomain interfaces.
The desired primal solution $u$ is recovered from the dual solution $\p$ of~\cref{dual_ROF_old} by the following relation:
\begin{equation*}
u = f  + \frac{1}{\alpha} \div \p.
\end{equation*}
Faster algorithms for solving \cref{dual_ROF_old} were developed in~\cite{BT:2009, Nesterov:2005}.

In the existing works~\cite{Chambolle:2004, CTWY:2015, HL:2015, LN:2017} for~\cref{dual_ROF_old}, the problems were discretized in the finite difference framework.
Each pixel in an image was treated as a discrete point on a grid, and the dual variable was considered as a vector-valued function on the grid.
The discrete gradient and divergence operators were defined by finite difference approximations of the continuous gradient and divergence operators.
In this paper, we propose a finite element discretization for \cref{dual_ROF_old}, which is more suitable for the DDMs than the existing ones.
Each pixel in an image is treated as a square finite element and the problem \cref{dual_ROF} is discretized by using the conforming lowest order Raviart--Thomas element~\cite{RT:1977}.

Based on the proposed discretization, we propose a primal DDM which is similar to the classical Schur complement method for the second order elliptic problems.
Eliminating the interior degrees of freedom in each subdomain yields an equivalent minimization problem to the full dimension problem.
The functional of the resulting minimization problem has enough regularity to adopt the FISTA~\cite{BT:2009}.
Thus, the proposed primal DDM achieves $O(1/n^2)$ convergence, and to the best of our knowledge, it is the best rate among the existing DDMs for the ROF model.
In addition, we propose a primal-dual DDM based on an equivalent saddle point problem.
The continuity of a solution on the subdomain interfaces is enforced by the method of Lagrange multipliers as in \cite{DCT:2016, FLP:2000, FR:1991}, and it yields an equivalent saddle point problem of the original variable (primal) and the Lagrange multipliers (dual).
The local problems for the proposed primal-dual DDM can be solved at a linear convergence rate, so that the method becomes very fast.

The rest of the paper is organized as follows.
In \cref{Sec:dual_ROF}, a conforming discretization of the dual ROF model with a Raviart--Thomas finite element space is introduced.
A primal DDM based on an equivalent minimization problem on the subdomain interfaces is presented in \cref{Sec:primal_DD}.
A primal-dual DDM based on an equivalent saddle point problem is considered in \cref{Sec:pd_DD}.
We present numerical results for the proposed methods in various settings in \cref{Sec:numerical}.
Finally, we conclude the paper with some remarks in \cref{Sec:conclusion}.

\section{The Dual ROF Model}
\label{Sec:dual_ROF}

\subsection{Preliminaries}
We review some preliminaries about the dual ROF model.
The space $\Hdiv$ is defined as
\begin{equation*}
\Hdiv = \left\{ \p \in (L^2 (\Omega))^2 : \div \p \in L^2 (\Omega) \right\}.
\end{equation*}
It is a Hilbert space equipped with an inner product
\begin{equation*}
\left< \p, \q \right>_{\Hdiv} = \intO \p \cdot \q \,dx + \intO \div \p \div \q \,dx,
\end{equation*}
and its induced norm is called the $\Hdiv$ graph norm.
A remarkable property of $\Hdiv$ is that, for a vector function $\p \in \Hdiv$, the normal component $\p \cdot \n$ on $\partial \Omega$ is well-defined \cite{BBF:2013, GR:2012}.
We define $\Hzdiv$ as the subspace of $\Hdiv$ with vanishing normal component on $\partial \Omega$.
It can be shown that the space $\Hzdiv$ is the closure of $(C_0^{\infty}(\Omega))^2$ in the $\Hdiv$ graph norm~\cite{Monk:2003}.
Thus, it is natural to consider the following alternative formulation of~\cref{dual_ROF_old} using $\Hzdiv$ as the solution space:
\begin{equation}
\label{dual_ROF}
\min_{\p \in \Hzdiv} \left\{ \J(\p) := \frac{1}{2\alpha} \int_{\Omega} (\div \p + \alpha f)^2 \,dx \right\} \hspace{0.5cm}
\textrm{subject to } |\p| \leq 1.
\end{equation}
We notice that this formulation was also considered in \cite{CTWY:2015}.

\subsection{Finite Element Discretizations}
A digital image consists of a number of rows and columns of pixels, holding values representing the intensity at a specific point.
We regard each pixel as a unit square and an image as a piecewise constant function in which each piece is a single pixel.
In this sense, we regard each pixel in a digital image as a square finite element whose side length equals~1.
Let $\T$ be the collection of all elements in $\Omega$, i.e. pixels.
We define the space $X$ for the image by
\begin{equation*}
X  = \left\{ u \ \in L^2 (\Omega) : u|_{T}\textrm{ is constant } \forall T \in \T \right\}.
\end{equation*}
Then it is clear that $X \subset BV(\Omega )$, which means that the discretization is conforming.
Each degree of freedom of $X$ lies in an element (see \cref{Fig:dofs}(a)), and its corresponding basis function is
\begin{equation*}
\phi_T (x) = \begin{cases} 1 & \textrm{ if } x \in T , \\ 0 & \textrm{ if } x \not\in T , \end{cases} \hspace{0.5cm} T \in \T.
\end{equation*}
For $u \in  X$ and $T \in \T$, let $(u)_T$ denote the degree of freedom of $u$ associated with the basis function $\phi_T$.
With a slight abuse of notation, let $\T$ also indicate the set of indices of the basis functions for $X$;
then we can represent $u$ by
$$u = \sum_{T \in \T} (u)_T \phi_T.$$

\begin{figure}[]
\centering
\subfloat[][Degrees of freedom for $X$]{ \includegraphics[height=3.8cm]{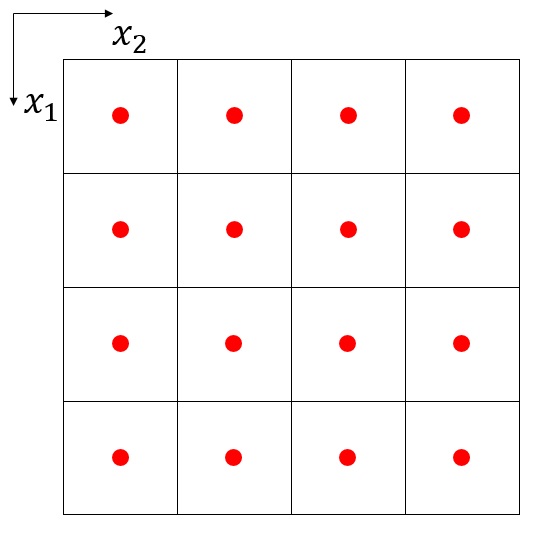} }
\hspace{1.5cm}
\subfloat[][Degrees of freedom for $Y$]{ \includegraphics[height=3.8cm]{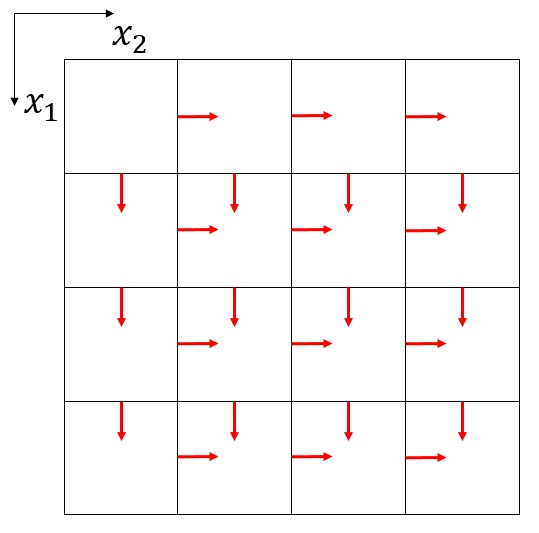} }
\caption{Degrees of freedom for the spaces $X$ and $Y$}
\label{Fig:dofs}
\end{figure} 

It is natural to determine the space $Y$ for the dual variable $\p$ such that the divergence of each element in $Y$ is in $X$.
A suitable choice to meet this condition is the lowest order Raviart--Thomas elements~\cite{RT:1977}.
We define $Y$ by
\begin{equation*}
Y = \left\{ \q \in \Hzdiv : \q|_{T} \in \mathcal{RT}_0(T) \hspace{0.2cm}\forall T \in \T \right\},
\end{equation*}
where $\mathcal{RT}_0(T)$ is the collection of the vector functions $\q$:~$T \rightarrow \R^2$ of the form
$$\q (x_1 , x_2) = \begin{bmatrix} a_1 + b_1 x_1 \\ a_2 + b_2 x_2 \end{bmatrix}.$$
In order for a piecewise $\mathcal{RT}_0(T)$-function to be in $\Hzdiv$, a particular condition on the element interfaces should be satisfied, which is given in the following proposition~\cite{Nedelec:1980}.

\begin{proposition}
\label{Prop:FEM_interface}
A vector function $\q$\emph{:} $\Omega \rightarrow \R^2$ is in $\Hdiv$ if and only if
the restriction of $\q$ to each $T \in \T$ is in $H(\div ; T)$, and
for each common edge $e = \bar{T_1} \cap \bar{T_2}$, we have
$$ \q \cdot \n |_{T_1} + \q \cdot \n |_{T_2} = 0 \textrm{ on } e, $$
where $\n |_{T_i}$ is the outer normal to $\partial T_i$ on $e$, $i= 1,2$, so that $\n|_{T_1} = - \n|_{T_2}$.
\end{proposition}

\Cref{Prop:FEM_interface} gives a natural way to choose the degrees of freedom of the space $Y$.
Let $\q \in Y$.
Then the value of $\q \cdot \n$ is well-defined on each common edge of elements,
where the direction of $\n$ is chosen as in \cref{Fig:dofs}(b).
Therefore, we choose the degrees of freedom of $Y$ by the values of $\q \cdot \n$ on the element interfaces.

To construct the corresponding basis functions, we consider a reference square $T_{\mathrm{ref}} = [0, 1]^2$.
The outer normal component of a basis function $\bm{\psi}_{\mathrm{ref}}$ has the value $1$ on one edge, say $x=1$, and $0$ on the other edges.
Such $\bm{\psi}_{\mathrm{ref}}$ is unique and given by $\bm{\psi}_{\mathrm{ref}}(x_1 , x_2 ) = (x_1 , 0)$.
Similarly, the other basis functions on $T_{\mathrm{ref}}$ are given by $(1-x_1, 0)$, $(0, x_2)$, and $(0, 1-x_2)$.

Now, let $\I$ be the set of indices of the basis functions for $Y$, and let $\left\{\bm{\psi}_i \right\}_{i \in \I}$ be the basis.
Also, for $\p \in Y$ and $i \in \I$, let $(\p)_i$ denote the degree of freedom of $\p$ associated with the basis function $\bm{\psi}_i$; then we can write
$$ \p = \sum_{i \in \I} {(\p)_i \bm{\psi}_i}.$$

Next, we determine the norms and the inner products for which $X$ and $Y$ will be equipped.
In $X$, the $L^2 (\Omega)$-inner product agrees with the Euclidean inner product, so it is natural to choose the inner product as
$$ \left< u, v \right>_X = \int_{\Omega} {uv \,dx} = \sum_{T \in \T} {(u)_T (v)_T} $$
and the norm as its induced norm $$\| u \|_X^2 = \left< u, u \right>_X.$$
We set the inner product for $Y$ by the usual Euclidean inner product
$$ \left< \p, \q \right>_Y = \sum_{i \in \I} {(\p)_i (\q)_i} $$
and the norm by its induced norm $$\| \p \|_Y^2  = \left< \p, \p \right>_Y.$$

\begin{remark}
We equipped $Y$ with not the $(L^2 (\Omega))^2$-inner product but the Euclidean inner product.
The reason is that if we equip $Y$ with the $(L^2 (\Omega))^2$-inner product, then the $(L^2 (\Omega))^2$-mass matrix occurs in the resulting algorithms, making computation more cumbersome.
In the following, we prove that using the Euclidean inner product instead of the $(L^2 (\Omega))^2$-inner product does not affect both the quality of image denoising and the rate of convergence.

Assume that the image size is $n = M \times N$.
Consider an $n \times n$ symmetric tridiagonal matrix~$\trid_n (\alpha , \beta)$ whose diagonal entries are~$\alpha$ and off-diagonal entries are~$\beta$.
Under an appropriate ordering of the degrees of freedom of~$Y$, one can see that the $(L^2 (\Omega))^2$-mass matrix is a block-diagonal matrix composed of~$N$ $\trid_{M-1} (\frac{2}{3}, \frac{1}{6})$-blocks and $M$ $\trid_{N-1} (\frac{2}{3}, \frac{1}{6})$-blocks.
Hence, all eigenvalues are
\begin{equation*}
\frac{2}{3} + \frac{1}{3} \cos \left(\frac{k\pi}{M}\right), \hspace{0.5cm} k=1, \ldots ,M-1,
\end{equation*}
and
\begin{equation*}
\frac{2}{3} + \frac{1}{3} \cos \left(\frac{k\pi}{N}\right), \hspace{0.5cm} k=1, \ldots ,N-1.
\end{equation*}
See section~C.7 in~\cite{LeVeque:2007} for details.
The $(L^2 (\Omega))^2$-mass matrix is spectrally equivalent to the identity matrix which can be obtained from the $(L^2 (\Omega))^2$-mass matrix by diagonal lumping with proper scaling.
Therefore, one can conclude that the overall performance remains the same even if we use the Euclidean inner product.
\end{remark}

For a pixel $T = T_{ij} \in \T$ on the $i$th row and the $j$th column of the $M \times N$ image, let $\iota_{T, 1} \in \I$ be the index corresponding to the degree of freedom of~$Y$ on the edge shared by~$T_{ij}$ and~$T_{i+1,j}$.
Similarly, let $\iota_{T, 2} \in  \I$ be the one on the edge shared by~$T_{ij}$ and $T_{i, j+1}$.
To treat the inequality constraints in~\cref{dual_ROF}, for $1< p < \infty$, we define the subset $C^p$ of $Y$ by
\begin{equation}
\label{Cp}
C^p = \left\{ \p \in Y : |(\p)_{\iota_{T, 1}}|^{q} + |(\p)_{\iota_{T, 2}}|^{q} \leq 1 \hspace{0.2cm} \forall T \in \T \right\},
\end{equation}
where $q$ is the H\"{o}lder conjugate of $p$ and the convention $(\p)_{\iota_{T_{M, j}, 1}} = (\p)_{\iota_{T_{i, N}, 2}} = 0$ is adopted.
Also, for $p=1$, we define
\begin{equation}
\label{C1}
C^1 = \left\{ \p \in Y : |(\p)_i| \leq 1 \hspace{0.2cm} \forall i \in \I \right\}.
\end{equation}
Clearly, for $1 \leq p < \infty$, $C^p$ is nonempty and convex.
The orthogonal projection of $\p \in Y$ onto $C^p$ can be easily computed by
\begin{equation}
\label{proj_Cp}
(\proj_{C^p} \p)_{\iota_{T, k}} = \frac{(\p)_{\iota_{T, k}}}{\left( |(\p)_{\iota_{T, 1}}|^{q} + |(\p)_{\iota_{T, 2}}|^{q} \right)^{\frac{1}{q}}}
\hspace{0.5cm} \forall T \in \T, \hspace{0.1cm} k=1,2
\end{equation}
for $1<p<\infty$ and
\begin{equation}
\label{proj_C1}
(\proj_{C^1} \p)_i = \frac{(\p)_i}{\max \left\{ 1, |(\p)_i| \right\}} \hspace{1cm} \forall i \in \I
\end{equation}
for $p = 1$.

Finally, we are ready to state a finite element version of problem \cref{dual_ROF}:
\begin{equation}
\label{d_dual_ROF}
\min_{\p \in Y} \J(\p) + \chi_{C^p} (\p),
\end{equation}
where $\chi_{C^p}$ is the characteristic function of $C^p$ which is defined as
\begin{equation*}
\chi_{C^p} (\p) = \begin{cases}0 & \textrm{ if } \p \in C^p,\\ \infty & \textrm{ if } \p \not\in C^p. \end{cases}
\end{equation*}
We provide a relation between \cref{d_dual_ROF} and the conventional finite difference discretization of the ROF model.

\begin{figure}[]
\centering
\subfloat[][$\div$]{ \includegraphics[height=3.5cm]{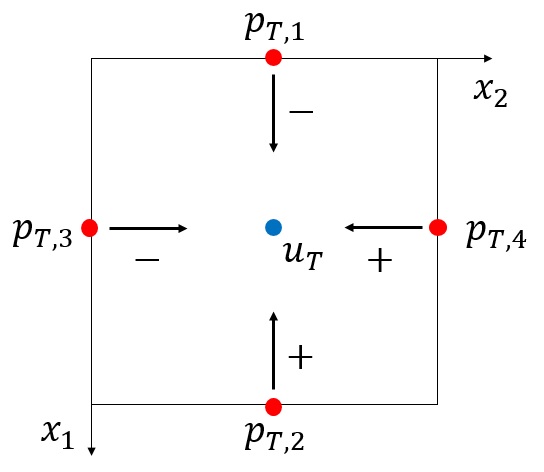} }
\hspace{1.2cm}
\subfloat[][$\div^*$]{ \includegraphics[height=3.5cm]{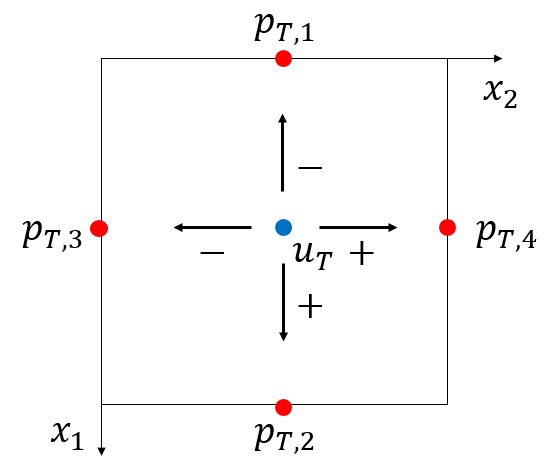} }
\caption{Action of the operators $\div$ and $\div^*$ on an element}
\label{Fig:div}
\end{figure}

\begin{theorem}
\label{Thm:equiv}
Let $\p^* \in Y$ be a solution of~\cref{d_dual_ROF}.
If we identify $X$ with the Euclidean space of the functions from the $M \times N$ discrete points $[1, ..., M] \times [1, ..., N]$ into~$\mathbb{R}$, then $u^* = f + \frac{1}{\alpha} \div \p^*$ is a solution of the finite difference ROF model
\begin{equation*}
\min_{u \in X } \frac{\alpha}{2} \| u - f \|_2^2 + \| | D u|_p \|_1, \hspace{0.5cm} (1 \leq p < \infty)
\end{equation*}
where $D u$ is the forward finite difference operator
\begin{eqnarray*}
(Du)_{ij}^1 &=& \begin{cases} u_{i+1, j} - u_{ij} & \textrm{ if } i=1, ..., M-1, \\ 0 & \textrm{ if } i = M,\end{cases} \\
(Du)_{ij}^2 &=& \begin{cases} u_{i, j+1} - u_{ij} & \textrm{ if } j=1, ..., N-1, \\ 0 & \textrm{ if } j = N\end{cases}
\end{eqnarray*}
and $(|Du|_p)_{ij} = \left( |(Du)_{ij}^1|^p + |(Du)_{ij}^2|^p \right)^{\frac{1}{p}}$.
\end{theorem}
\begin{proof}
By the primal-dual equivalence, $u^*$ is a solution of the Fenchel--Rockafellar dual of~\cref{d_dual_ROF} given by
\begin{equation*}
\min_{u \in X} \left\{ \frac{\alpha}{2} \intO (u-f)^2 \,dx + \sup_{\p \in C^p} \intO u \div \p \,dx \right\}.
\end{equation*}
Then, we have
\begin{align*}
 \frac{\alpha}{2} \intO (u-f)^2 \,dx + \sup_{\p \in C^p} \intO u \div \p \,dx &= \frac{\alpha}{2} \| u - f \|_2^2 + \sup_{\p \in C^p} \left< u , \div \p \right>_X \\
&= \frac{\alpha}{2} \| u - f \|_2^2 + \sup_{\p \in C^p} \left< \div^* u , \p \right>_Y,
\end{align*}
where $\div^*$:~$X \rightarrow Y$ is defined as
\begin{equation*}
\left< \div^* u , \p \right>_Y = \left< u , \div \p \right>_X \hspace{0.5cm} \forall u \in X, \p \in Y.
\end{equation*}
Observe that the $\div^*$ operator acts like the minus finite difference operator (see \cref{Fig:div}(b)).
Indeed, we can see that
\begin{eqnarray*}
(\div^* u)_{\iota_{T, 1}} &=& u_{ij} - u_{i+1, j} = - (Du)_{ij}^1\\
(\div^* u)_{\iota_{T, 2}} &=& u_{ij} - u_{i, j+1} = - (Du)_{ij}^2
\end{eqnarray*}
for $T = T_{ij} \in \T$ with the convention $u_{Mj} - u_{M+1, j} = u_{iN} - u_{i,N+1} = 0$.
Assume $1<p<\infty$, $\frac{1}{p} + \frac{1}{q} = 1$, and take any $\p \in C^p$.
Then, by the duality between the spaces $l^p$ and $l^q$ in each pixel, we get
\begin{align*}
\sup_{\p \in C^p} \left< \div^* u , \p \right>_Y &= \sum_{T = T_{ij} \in \T} \sup_{|(\p)_{\iota_{T, 1}}|^{q}  + |(\p)_{\iota_{T, 2}}|^{q} \leq 1} \left[ (Du)_{ij}^1 (\p)_{\iota_{T, 1}} + (Du)_{ij}^2 (\p)_{\iota_{T, 2}} \right] \\
&= \sum_{T = T_{ij} \in \T} \left[ |(Du)_{ij}^1|^p + |(Du)_{ij}^2|^p \right]^{\frac{1}{p}} = \| |Du|_p \|_1,
\end{align*}
which concludes the proof.
The case for $p=1$ is straightforward.
\end{proof}

\Cref{Thm:equiv} means that, by choosing the set $C^p$ appropriately, the finite element model~\cref{d_dual_ROF} can express various versions of discrete total variation, for example, an anisotropic one for $p=1$ and an isotropic one for $p=2$.
Hereafter, for the sake of simplicity, we treat the case for $p=1$ only; generalization to the other cases is straightforward.
We drop the superscript and write $C = C^1$.

Next, note that the divergence operator in the continuous setting is well-defined on $Y$, and its image is contained in $X$.
That is, the divergence of a function in $Y$ is piecewise constant.
This means that we do not need to define a discrete divergence operator as in the preceding researches,
and some good properties from the continuous setting are inheritable to our discretization.
For instance, for a nonoverlapping domain decomposition $\left\{ \Omega_s \right\}_{s=1}^{\N}$ of $\Omega$ and $\p \in Y$,
the following \textit{splitting property} of $\J(\p)$ holds:
\begin{equation}
\label{splitting}
\frac{1}{2\alpha} \intO (\div \p + \alpha f)^2 \,dx = \sum_{s=1}^{\N} \frac{1}{2\alpha} \intOs ( \div(\p|_{\Omega_s}) + \alpha f) ^2 \,dx .
\end{equation}
\Cref{splitting} will be our main tool in designing the DDMs in \cref{Sec:primal_DD,Sec:pd_DD}.

\begin{remark}
The discrete divergence operator proposed in \cite{Chambolle:2004, LN:2017} does not satisfy~\cref{splitting}, which was designed in the finite difference framework.
\end{remark}

\subsection{Solvers for the Finite Element ROF Model}
The proposed discrete problem~\cref{d_dual_ROF} can adopt the existing solvers for the total variation minimization
using either dual approaches~\cite{BT:2009, Chambolle:2004} or primal-dual approaches~\cite{CP:2011}.
We give some results about~\cref{d_dual_ROF} which help to set the parameters for the solvers.

\begin{proposition}
\label{Prop:div_norm}
The operator norm of $\div$\emph{:} $Y \rightarrow X$ has a bound such that $\| \div \|_{Y \rightarrow X}^2 \leq 8$.
\end{proposition}
\begin{proof}
Fix $\p \in Y$.
For a pixel $T \in \T$, let $p_{T, 1}$, $p_{T, 2}$, $p_{T, 3}$, and $p_{T, 4}$ be the degrees of freedom of $\p$ on the top, bottom, left, and right edges of $T$, respectively (see \cref{Fig:div}).
We may set $p_{T, j}$ by $0$ if it is on $\partial \Omega$ for some $j$.
Then, we have
\begin{align*}
(\div \p)_T^2 &= (-p_{T, 1} + p_{T, 2} - p_{T, 3} + p_{T,4})^2\\
&\leq 4(p_{T, 1}^2 + p_{T, 2}^2 + p_{T, 3}^2 + p_{T,4}^2 ).
\end{align*}
Summation over all $T \in \T$ yields
\begin{align*}
\| \div \p \|_X^2 = \sum_{T \in \T} {(\div \p)_T^2} &\leq 4 \sum_{T \in \T} {(p_{T, 1}^2 + p_{T, 2}^2 + p_{T, 3}^2 + p_{T,4}^2 )}\\
&\leq 8 \sum_{i \in \I} {(\p)_i^2} = 8 \| \p \|_Y^2. 
\end{align*}
For the second inequality, the fact that every edge is shared by at most two elements is used.
Therefore, $\| \div \|_{Y \rightarrow X}^2 \leq 8$.
\end{proof}

\begin{proposition}
\label{Prop:d_dual_ROF_Lipschitz}
The gradient of $\J (\p)$ is given by
$$\nabla \J (\p) = \frac{1}{\alpha} \div^* ( \div\p + \alpha f)$$
and it is Lipschitz continuous with a Lipschitz constant $8/\alpha$.
\end{proposition}
\begin{proof}
Take any $\p \in Y$, and let $\q \in Y$ with $\|\q \|_Y = 1$, and $h > 0$.
Then we have
\begin{align*}
\left| \J (\p + h\q) - \J(\p) - \left< \frac{1}{\alpha} \div^* ( \div\p  + \alpha f ), h\q \right>_Y \right|
&= \frac{h^2}{2\alpha} \int_{\Omega} (\div \q)^2 \,dx \\
&\leq \frac{h^2}{2\alpha} \| \div \|_{Y \rightarrow X}^2 \| \q \|_Y^2 \leq \frac{4h^2}{\alpha} .
\end{align*}
Therefore, $\nabla \J (\p) = \frac{1}{\alpha} \div^* (\div\p + \alpha f )$.
Furthermore, for any $\p$, $\q \in Y$,
\begin{align*}
\left\| \nabla \J (\p) - \nabla \J (\q) \right\|_Y &= \left\| \frac{1}{\alpha} \div^* \left( \div (\p - \q) \right) \right\|_Y \\
&\leq \frac{1}{\alpha} \| \div \|_{Y \rightarrow X}^2 \| \p - \q \|_Y \leq \frac{8}{\alpha} \| \p - \q \|_Y.
\end{align*}
In the last line, we used \cref{Prop:div_norm} to bound $\| \div \|_{Y \rightarrow X}$.
From the above computations, we conclude that $\nabla \J$ is Lipschitz continuous with a Lipschitz constant $8/\alpha$.
\end{proof}

We notice that the proof of \cref{Prop:div_norm} given here is essentially the same as the proof of Theorem~3.1 of~\cite{Chambolle:2004}.

\section{A Primal Domain Decomposition Method}
\label{Sec:primal_DD}
In this section, we propose a primal DDM for the proposed discretization which resembles the Schur complement method,
one of the most primitive nonoverlapping DDMs for second order elliptic problems.
We note that the method proposed in this section is not a DDM for the ``primal" total variation minimization problem,
but a ``primal" DDM for the ``dual" total variation minimization problem.
In the Schur complement method for second order elliptic problems, the degrees of freedom in the interior of the subdomains are eliminated so that only the degrees of freedom on the subdomain interfaces remain.
The remaining system on the subdomain interfaces is called the Schur complement system, and it is solved by an iterative solver like the conjugate gradient method.
Similarly, in the proposed method, the interior degrees of freedom are eliminated and we solve a resulting minimization problem on the subdomain interfaces.
Every finite-dimensional Hilbert space~$H$ appearing in \cref{Sec:primal_DD,Sec:pd_DD} is equipped with the Euclidean inner product~$\left< \cdot , \cdot \right>_H$ and the induced norm~$\| \cdot \|_H$.

\begin{figure}[]
\centering
\subfloat[][Primal DD]{ \includegraphics[height=3.8cm]{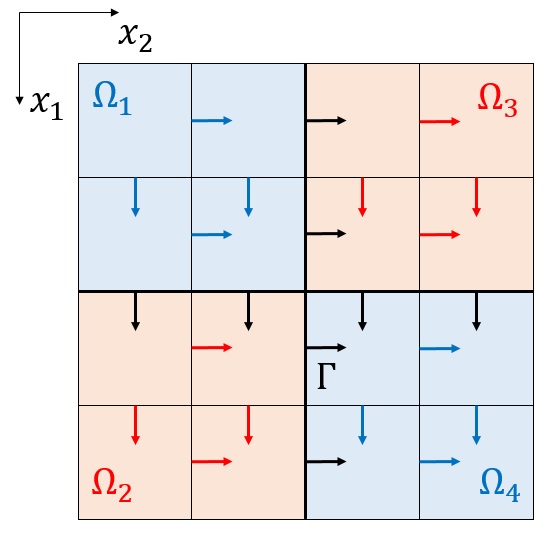} }
\hspace{2cm}
\subfloat[][Primal-dual DD]{ \includegraphics[height=4cm]{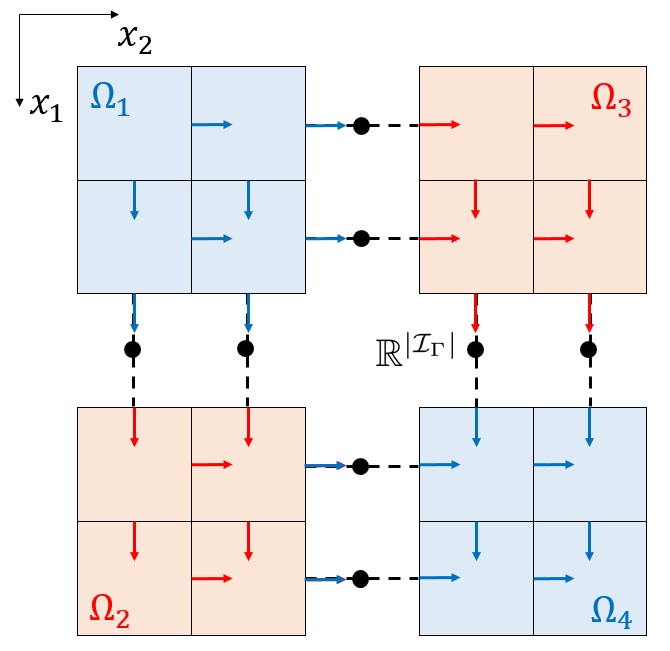} }
\caption{Primal and primal-dual domain decomposition}
\label{Fig:DD}
\end{figure}

We decompose the image domain $\Omega$ into $\N = N \times N$ disjoint square subdomains $\left\{ \Omega_s \right\}_{s=1}^{\N}$ in a checkerboard fashion (see \cref{Fig:DD}(a)).
From now on, the letters $s$ and $t$ stand for indices of subdomains, that is, $s$ and $t$ run from $1$ to $\N$.
We denote the outer normal to $\partial \Omega_s$ by $\n_s$.
For two adjacent subdomains $\Omega_s$ and $\Omega_t$ with $s < t$, let $\Gamma_{st} = \partial \Omega_s \cap \partial \Omega_t$ be the subdomain interface between them.
The subdomain interface $\Gamma_{st}$ is oriented in the way that the normal $\n_{st}$ to $\Gamma_{st}$ is given by $\n_{st} = \n_s = -\n_t$.
Also, we define the union of the subdomain interfaces $\Gamma$ by $\Gamma = \bigcup_{s<t} \Gamma_{st}$.

For the discrete setting, let $\T_s$ be the collection of all elements in $\Omega_s$.
We define the local dual function space $Y_s$ by
\begin{equation}
\label{Y_s}
Y_s = \left\{ \q_s \in  H_0(\div ; \Omega_s) : \q_s |_{T} \in \mathcal{RT}_0 (T) \hspace{0.2cm} \forall T \in \T_s  \right\}.
\end{equation}
Also, let~$\I_s$ be the set of indices of the basis functions for $Y_s$.
In addition, we set $Y_I$ by the direct sum of all local dual function spaces, that is,
\begin{equation*}
Y_I = \bigoplus_{s=1}^{\N} Y_s.
\end{equation*}
One can observe that, for $\p_I = \bigoplus_{s=1}^{\N} \p_s$ and $\q_I = \bigoplus_{s=1}^{\N} \q_s$, we have
\begin{equation*}
\left< \p_I , \q_I \right>_{Y_I} = \sum_{s=1}^{\N} \left< \p_s , \q_s \right>_{Y_s}.
\end{equation*}
Next, we denote $\I_{\Gamma}$ by the set of indices of degrees of freedom of $Y$ on $\Gamma$,
and define the interface function space $Y_{\Gamma}$ by
\begin{equation*}
Y_{\Gamma} = \mathrm{span} \left\{ \bm{\psi}_i \right\}_{i \in \I_{\Gamma}}.
\end{equation*}
The interface function space~$Y_\Gamma$ is equipped with the inner product defined by
\begin{equation*}
\left< \p_{\Gamma} , \q_{\Gamma} \right>_{Y_{\Gamma}} = \left< \p_{\Gamma} , \q_{\Gamma} \right>_Y
\end{equation*}
and its induced norm
\begin{equation*}
\| \p_{\Gamma} \|_{Y_{\Gamma}}^2 = \left< \p_{\Gamma} , \p_{\Gamma} \right>_{Y_{\Gamma}}.
\end{equation*}

As we readily see, $Y = Y_I  \oplus Y_{\Gamma}$.
For $\p \in Y$, there exists a unique decomposition
$$\p = \p_I \oplus \p_{\Gamma} = \left( \bigoplus_{s=1}^{\N} \p_s  \right) \oplus \p_{\Gamma}$$
with $\p_s \in Y_s$ and $\p_{\Gamma} \in Y_{\Gamma}$.
Thanks to the splitting property \cref{splitting}, we have
\begin{equation} \begin{split}
\label{primal_DD_splitting}
\J(\p) &= \frac{1}{2\alpha} \intO (\div \p + \alpha f)^2 \,dx \\
&= \sum_{s=1}^{\N} \frac{1}{2\alpha} \intOs (\div(\p_s + \p_{\Gamma}|_{\Omega_s}) + \alpha f)^2 \,dx.
\end{split} \end{equation}

To treat the inequality constraints, as we did in~\cref{C1}, we define the subset $C_s$ of $Y_s$ by
\begin{equation}
\label{C_s}
C_s = \left\{ \p_s \in Y_s : |(\p_s)_i| \leq 1 \hspace{0.2cm} \forall i \in \I_s \right\},
\end{equation}
and we set $C_I$ as the direct sum of all $C_s$'s:
\begin{equation*}
\label{C_I}
C_I = \bigoplus_{s=1}^{\N} C_s.
\end{equation*}
In addition, let $C_{\Gamma}$ be the subset of $Y_{\Gamma}$ satisfying the inequality constraints:
\begin{equation*}
\label{C_Gamma}
C_{\Gamma} = \left\{ \p_{\Gamma} \in Y_{\Gamma} : |(\p_{\Gamma})_i| \leq 1\hspace{0.2cm}  \forall i \in \I_{\Gamma} \right\}.
\end{equation*}
Similarly to~\cref{proj_C1}, the projections onto $C_s$ and $C_{\Gamma}$ can be computed by the pointwise Euclidean projection:
\begin{subequations}
\begin{equation}
\label{proj_C_s}
(\proj_{C_s} \p_s )_i = \frac{(\p_s)_i}{\max \left\{ 1, |(\p_s)_i| \right\}} \hspace{0.5cm} \forall i \in \I_s,
\end{equation}
\begin{equation}
(\proj_{C_{\Gamma}} \p_{\Gamma} )_i = \frac{(\p_{\Gamma})_i}{\max \left\{ 1, |(\p_{\Gamma})_i| \right\}} \hspace{0.5cm} \forall i \in \I_{\Gamma}.
\end{equation}
\end{subequations}

Now, for $\p_{\Gamma} \in C_{\Gamma}$, we consider the following minimization problem:
\begin{equation}
\label{primal_DD_harmonic}
\min_{\p_I \in Y_I} \left\{ \J (\p_I \oplus \p_{\Gamma}) + \chi_{C_I}(\p_I) \right\}.
\end{equation}
We note that, with the help of~\cref{primal_DD_splitting}, a solution of \cref{primal_DD_harmonic} can be obtained by solving 
\begin{equation}
\label{primal_DD_harmonic_local}
\min_{\p_s \in Y_s} \left\{ \frac{1}{2\alpha}\intOs (\div(\p_s + \p_{\Gamma}|_{\Omega_s}) + \alpha f)^2 \,dx + \chi_{C_s}(\p_s) \right\}
\end{equation}
and taking the direct sum of the solutions of~\cref{primal_DD_harmonic_local} over $s=1, ..., \N$.
The local problem~\cref{primal_DD_harmonic_local} can be solved independently in each subdomain.
That is, no communications among processors are required so that the resulting algorithm becomes suitable for parallel computation.
With a slight abuse of notation, we denote a solution of~\cref{primal_DD_harmonic} by $\Hp \in C_I$.
Although $\Hp$ is not unique in general, $\div (\Hp)$ is uniquely determined and we will deal with $\div (\Hp)$ only.

Finally, we present the minimization problem for the proposed primal DDM:
\begin{equation}
\label{primal_DD}
\min_{\p_{\Gamma} \in Y_{\Gamma}} \J_{\Gamma}(\p_{\Gamma}) + \chi_{C_{\Gamma}} (\p_{\Gamma}),
\end{equation}
where the functional $\J_{\Gamma}(\p_{\Gamma})$ on $Y_{\Gamma}$ is defined as
\begin{equation}
\label{J_Gamma}
\J_{\Gamma} (\p_{\Gamma}) = \J (\Hp \oplus \p_{\Gamma}).
\end{equation}
The functional $\J_{\Gamma} (\p_{\Gamma})$ can be regarded as the result of elimination of interior degrees of freedom $\p_I$ from $\J (\p)$.
The same technique is widely used in DDMs for second order elliptic problems.
The following proposition shows a relation between~\cref{d_dual_ROF} and~\cref{primal_DD}.

\begin{proposition}
\label{Prop:primal_DD_equiv}
If $\p^* \in Y$ is a solution of~\cref{d_dual_ROF}, then $\p_{\Gamma}^* = \p^* |_{Y_{\Gamma}}$ is a solution of~\cref{primal_DD}.
Conversely, if~$\p_{\Gamma}^* \in Y_{\Gamma}$ is a solution of~\cref{primal_DD}, then $\p^* = \Hp^* \oplus \p_{\Gamma}^*$ is a solution of~\cref{d_dual_ROF}.
\end{proposition}
\begin{proof}
Let $\p^* \in Y$ be a solution of~\cref{d_dual_ROF} and $\p_{\Gamma}^* = \p^* |_{Y_{\Gamma}}$.
Clearly, $\p_{\Gamma}^* \in C_{\Gamma}$.
We show that $\J_{\Gamma} (\p_{\Gamma}) \geq \J_{\Gamma} (\p_{\Gamma}^*)$ for all $\p_{\Gamma} \in C_{\Gamma}$.
Take any $\p_{\Gamma} \in C_{\Gamma}$.
Then, by the minimization property of $\p^*$ with respect to~\cref{d_dual_ROF}, we have
\begin{equation*}
\J_{\Gamma} (\p_{\Gamma}) = \J (\Hp \oplus \p_{\Gamma}) \geq \J(\p^*).
\end{equation*}
Also, by the minimization property of $\mathcal{H}_I$ with respect to~\cref{primal_DD_harmonic}, we have
\begin{align*}
\J(\p^*) &= \J(\p^* |_{Y_I} \oplus \p_{\Gamma}^*) \\
&\geq \J(\Hp^* \oplus \p_{\Gamma}^*) = \J_{\Gamma} (\p_{\Gamma}^*).
\end{align*}
Therefore, $\J_{\Gamma} (\p_{\Gamma}) \geq \J_{\Gamma} (\p_{\Gamma}^*)$, so that $\p_{\Gamma}^*$ is a solution of~\cref{primal_DD}.

Conversely, let $\p_{\Gamma}^* \in Y_{\Gamma}$ be a solution of~\cref{primal_DD} and $\p^* = \Hp^* \oplus \p_{\Gamma}^* \in C$.
It suffices to show that $\J(\p) \geq \J(\p^*)$ for all $\p \in C$.
Take any $\p \in C$.
By the minimization property of $\mathcal{H}_I$ with respect to~\cref{primal_DD_harmonic}, we have
\begin{align*}
\J(\p) &= \J(\p |_{Y_I} \oplus \p|_{Y_{\Gamma}}) \\
&\geq \J(\mathcal{H}_I \p |_{Y_{\Gamma}} \oplus \p|_{Y_{\Gamma}}) = \J_{\Gamma} (\p|_{Y_{\Gamma}}),
\end{align*}
while
\begin{equation*}
\J_{\Gamma} (\p|_{Y_{\Gamma}}) \geq \J_{\Gamma} (\p_{\Gamma}^*) = \J(\p^*)
\end{equation*}
by the minimization property of $\p_{\Gamma}^*$ with respect to~\cref{primal_DD}.
Therefore, $\p^*$ is a solution of~\cref{d_dual_ROF}.
\end{proof}

By \cref{Prop:primal_DD_equiv}, it is enough to solve~\cref{primal_DD} to obtain a solution of~\cref{d_dual_ROF}.
As we noted in~\cref{primal_DD_harmonic_local}, \cref{primal_DD} has an intrinsic domain decomposition structure,
so that the parallelization of the algorithm at the subdomain level is straightforward regardless of the choice of solver for the minimization problem.
In this paper, we adopt FISTA~\cite{BT:2009} as the solver for~\cref{primal_DD}, which is known to have $O(1/n^2)$ convergence.
To the best of our knowledge, there have been no DDMs for the ROF model with convergence rate better than~$O(1/n^2)$.
In particular, Chang~et al.~\cite{CTWY:2015} showed that the subspace correction methods for the dual ROF model has the theoretical convergence rate~$O(1/n)$ even in the overlapping domain decomposition case.

To show the suitability of FISTA for~\cref{primal_DD}, it should be ensured that the functional $\J_{\Gamma} (\p_{\Gamma})$ in \cref{J_Gamma} is differentiable and its gradient is Lipschitz continuous.
The following lemmas are ingredients for showing such regularity of $\J_{\Gamma} (\p_{\Gamma})$.
At first, \cref{Lem:primal_DD_div_norm} tells that the norm bound of the $\div$ operator can be improved from \cref{Prop:div_norm} if its domain is restricted to~$Y_{\Gamma}$.

\begin{lemma}
\label{Lem:primal_DD_div_norm}
Assume that each subdomain consists of at least $2 \times 2$ pixels.
Then, the operator norm of $\div$\emph{:} $Y_{\Gamma} \rightarrow X$ has a bound such that $\| \div \|_{Y_{\Gamma} \rightarrow X}^2 \leq 4$.
\end{lemma}
\begin{proof}
Fix $\p_{\Gamma} \in Y_{\Gamma}$ and let $\p = \mathbf{0}_I \oplus \p_{\Gamma} \in Y$, which is an extension of $\p_{\Gamma}$ to $Y$.
We clearly have
$$ \div \p = \div \p_{\Gamma}.$$
For a pixel $T \in \T$, similarly to \cref{Prop:div_norm}, let $p_{T, 1}$, $p_{T, 2}$, $p_{T, 3}$, and $p_{T, 4}$ be the degrees of freedom of $\p$ on the top, bottom, left, and right edges of $T$, respectively.
Since~$\partial T \cap \Gamma$ consists of at most two element edges (when $T$ is at a subdomain corner),
at most two of $\p_{T, i}$'s are nonzero.
Thus, we have
\begin{align*}
(\div \p)_T^2 &= (-p_{T, 1} + p_{T, 2} - p_{T, 3} + p_{T, 4})^2\\
&\leq 2(p_{T, 1}^2 + p_{T, 2}^2 + p_{T, 3}^2 + p_{T, 4}^2 ),
\end{align*}
where we use the Cauchy--Schwarz inequality.
Summation over all $T \in T$ yields
\begin{align*}
\| \div \p \|_X^2 = \sum_{T \in \T} {(\div \p)_T^2} &\leq 2 \sum_{T \in \T} {(p_{T, 1}^2 + p_{T, 2}^2 + p_{T, 3}^2 + p_{T, 4}^2 )}\\
&\leq 4 \sum_{i \in \I} {(\p)_i^2} \\
&= 4 \sum_{i \in \I_{\Gamma}} (\p_{\Gamma})_i^2= 4 \| \p_{\Gamma} \|_{Y_{\Gamma}}^2. 
\end{align*}
Therefore, $\| \div \|_{Y_{\Gamma} \rightarrow X}^2 \leq 4$.
\end{proof}

Now we provide the main tool for showing the regularity of $\J_{\Gamma} (\p_{\Gamma})$, which is stated in a more general setting.
We note that \cref{Lem:smooth} can be regarded as a generalization of the smoothness property of the Moreau envelope~\cite{CP:2016}.

\begin{lemma}
\label{Lem:smooth}
Suppose that $H$, $H_1$, and $H_2$ are finite-dimensional Hilbert spaces.
Let $A$:~$H_1 \rightarrow H$, $B$:~$H_2 \rightarrow H$ be linear operators and $c \in H$.
Also, let $g$:~$H_2 \rightarrow \bar{\mathbb{R}}$ be a proper, convex, and lower semicontinuous functional.
Then a functional $F$:~$H_1 \rightarrow \mathbb{R}$ defined as
\begin{equation*}
F(x) = \min_{y \in H_2} \left\{ f(x, y) := \frac{1}{2} \| Ax + By + c \|_H^2 + g(y) \right\}
\end{equation*}
is differentiable and its gradient is given by
\begin{equation*}
\nabla F(x) = A^* (Ax + B y^*(x) + c),
\end{equation*}
where $y^* (x) = \argmin_{y \in H_2} f(x, y)$.
Furthermore, $\nabla F$ is Lipschitz continuous with modulus $L = \| A \|_{H_1 \rightarrow H}^2$.
\end{lemma}
\begin{proof}
For $x \in H_1$, let
\begin{equation*}
d(x) = A^* (Ax + By^* (x) + c).
\end{equation*}
One can easily verify that $d(x)$ is single-valued even though $y^*(x)$ may not be.

Take any $x_1$, $x_2 \in H_1$ and write $y_1 = y^* (x_1)$, $y_2 = y^* (x_2)$.
Then, by the minimization property of $y_1$, we get
\begin{equation} \begin{split}
\label{ub}
F(x_1) &= \frac{1}{2} \| Ax_1 + By_1 + c \|_H^2 + g(y_1) \\
&\leq \frac{1}{2} \| Ax_1 + By_2 + c \|_H^2 + g(y_2) \\
&= g(y_2) + \frac{1}{2} \| Ax_2 + By_2 + c \|_H^2 + \left< Ax_2 + By_2 + c , A(x_1 - x_2) \right>_H  \\
&\quad+ \frac{1}{2} \| A(x_1 - x_2) \|_H^2  \\
&\leq F(x_2) + \left< d(x_2) , x_1 - x_2 \right>_{H_1} + \frac{L}{2} \| x_1 - x_2 \|_{H_1}^2 . 
\end{split} \end{equation}
On the other hand, the optimality condition of $y_2$ reads as
\begin{equation*}
\label{y2_opt}
g(y) \geq g(y_2) + \left< Ax_2 + By_2 + c , B(y_2 - y)\right>_H \hspace{0.5cm} \forall y \in H_2.
\end{equation*}
Thus, it follows that
\begin{equation} \begin{split}
\label{lb_temp}
F(x_1) &= \frac{1}{2} \| Ax_1 + By_1 + c \|_H^2 + g(y_1) \\
&= g(y_1) + \frac{1}{2} \| Ax_2 + By_1 + c \|_H^2 + \left< Ax_2 + By_1 + c , A(x_1 - x_2) \right>_H \\
&\quad + \frac{1}{2} \| A(x_1 -x_2) \|_H^2 \\
&\geq g(y_2 ) + \left< Ax_2 + By_2 + c , B(y_2 - y_1) \right>_{H} + \frac{1}{2} \| Ax_2 + By_1 + c \|_H^2 \\
&\quad + \left< Ax_2 + By_1 + c , A(x_1 - x_2) \right>_H + \frac{1}{2} \| A(x_1 -x_2) \|_H^2 . 
\end{split} \end{equation}
By the vector identity
\begin{equation*}
\left< a +b , b \right> + \frac{1}{2} \| a \|_2^2 = \frac{1}{2} \| a +b \|_2^2 + \frac{1}{2} \| b \|_2^2,
\end{equation*}
equation~\cref{lb_temp} is written as
\begin{equation} \begin{split}
\label{lb}
F(x_1) &\geq g(y_2) + \frac{1}{2} \| Ax_2 + By_2 + c \|_H^2 + \frac{1}{2} \| B(y_1 - y_2) \|_H^2 \\ 
&\quad + \left< Ax_2 + By_1 + c , A(x_1 - x_2) \right>_H + \frac{1}{2} \| A(x_1 -x_2) \|_H^2 \\
&= F(x_2) + \frac{1}{2} \| B(y_1 - y_2) \|_H^2 + \left< Ax_2 + By_2 + c, A(x_1 - x_2) \right>_H \\
&\quad + \left< B(y_1 - y_2) , A(x_1 - x_2)\right>_H + \frac{1}{2} \| A(x_1 - x_2) \|_H^2 \\
&= F(x_2) + \left< d(x_2) , x_1 -x_2 \right>_{H_1} + \frac{1}{2} \| (Ax_1 + By_1 + c) - (Ax_2 + By_2 + c) \|_H^2 \\
&\geq F(x_2) + \left< d(x_2) , x_1 -x_2 \right>_{H_1} + \frac{1}{2L} \| d(x_1) - d(x_2) \|_{H_1}^2 . 
\end{split} \end{equation}
From \cref{ub,lb}, we conclude that $F$ is differentiable with $\nabla F = d$.

Now, it remains to show that $\nabla F$ is Lipschitz continuous.
Interchanging $x_1$ and $x_2$ in~\cref{lb} yields
\begin{equation}
\label{lb2}
F(x_2) \geq F(x_1) - \left< d(x_1) , x_1 - x_2 \right>_{H_1} + \frac{1}{2L} \| d(x_1) - d(x_2) \|_{H_1}^2.
\end{equation}
Summing~\cref{lb,lb2}, we obtain
\begin{align*}
\frac{1}{L} \| d(x_1) - d(x_2) \|_{H_1}^2 &\leq \left< d(x_1) - d(x_2) , x_1 - x_2 \right>_{H_1} \\
&\leq \| d(x_1) - d(x_2) \|_{H_1} \| x_1 - x_2 \|_{H_1},
\end{align*}
which means that $d$ is Lipschitz continuous with modulus~$L$.
\end{proof}

Now, we obtain the desired regularity result of $\J_{\Gamma} (\p_{\Gamma})$ as a direct consequence of \cref{Lem:smooth}.

\begin{corollary}
\label{Cor:smooth}
The gradient of $\J_{\Gamma} (\p_{\Gamma})$ is given by
$$\nabla \J_{\Gamma} (\p_{\Gamma}) = \frac{1}{\alpha} \div^* ( \div(\Hp \oplus \p_{\Gamma}) + \alpha f) |_{Y_{\Gamma}},$$
which is Lipschitz continuous with a Lipschitz constant $4/\alpha$.
\end{corollary} 
\begin{proof}
In \cref{Lem:smooth}, we set $H = X$, $H_1 = Y_{\Gamma}$, and $H_2 = Y_I$.
Taking $A = \div$:~$Y_{\Gamma} \rightarrow X$, $B = \div$:~$Y_I \rightarrow X$, and $g = \chi_{C_I}$ yields the conclusion.
In this case, we have $L = 4/\alpha$ due to \cref{Lem:primal_DD_div_norm}.
\end{proof}

\Cref{Cor:smooth} guarantees that FISTA is appropriate for~\cref{primal_DD}.
The proposed primal DDM for the dual ROF model is summarized in \cref{Alg:primal_DD}.

\begin{algorithm}[]
\caption{Primal DDM}
\begin{algorithmic}[]
\label{Alg:primal_DD}
\STATE Choose $L \geq 4$. Let $\q_{\Gamma}^{(0)} = \p_{\Gamma}^{(0)} = \0_{\Gamma}$ and $t_0 = 1$.  
\FOR{$n=0,1,2,...$}
\STATE $\displaystyle \Hq^{(n)} \in \argmin_{\q_I \in Y_I} \left\{ \J (\q_I \oplus \q_{\Gamma}^{(n)}) + \chi_{C_I}(\q_I) \right\}$
\STATE $\displaystyle \p_{\Gamma}^{(n+1)} = \proj_{C_{\Gamma}} \left( \q_{\Gamma}^{(n)} - \frac{1}{L} \div^* \left(\div (\Hq^{(n)} + \q_{\Gamma}^{(n)}) + \alpha f \right) \Big|_{Y_{\Gamma}}\right)$
\STATE $\displaystyle t_{n+1} = \frac{1 + \sqrt{1+4t_n^2}}{2}$
\STATE $\displaystyle \q_{\Gamma}^{(n+1)}  = \p_{\Gamma}^{(n+1)} + \frac{t_n - 1}{t_{n+1}}(\p_{\Gamma}^{(n+1)} - \p_{\Gamma}^{(n)})$
\ENDFOR
\end{algorithmic}
\end{algorithm}
As we noted in~\cref{primal_DD_harmonic_local}, $\Hq^{(n)}$ in \cref{Alg:primal_DD} can be obtained independently in each subdomain.
Indeed, $\Hq^{(n)} = \bigoplus_{s=1}^{\N} \q_s^{(n)}$ where $\q_s^{(n)}$ is a solution of
\begin{equation}
\label{primal_local}
\min_{\q_s \in Y_s} \left\{ \frac{1}{2\alpha} \intOs \left(\div(\q_s + \q_{\Gamma}^{(n)} |_{\Omega_s} ) + \alpha f \right)^2 \,dx  + \chi_{C_s}(\q_s)\right\}.
\end{equation}
Since $\q_{\Gamma}^{(n)} |_{\Omega_s}$ plays a role of only the essential boundary condition in~\cref{primal_local},
the existing solvers for the ROF model can be utilized to obtain $\q_s^{(n)}$ with little modification.
Convergence analysis for \cref{Alg:primal_DD} is straightforward~\cite{BT:2009}.

\begin{theorem}
\label{Thm:primal_DD}
Let $\{ \p_{\Gamma}^{(n)} \}$ be the sequence generated by \cref{Alg:primal_DD}, and let $\p_{\Gamma}^*$ be a solution of~\cref{primal_DD}.
Then for any $n \geq 1$, 
\begin{equation*}
\J_{\Gamma} (\p_{\Gamma}^{(n)}) - \J_{\Gamma}(\p_{\Gamma}^*) \leq \frac{2L \| \p_{\Gamma}^{(0)} - \p_{\Gamma}^* \|_{Y_{\Gamma}}^2}{(n+1)^2}.
\end{equation*} 
\end{theorem}

\section{A Primal-Dual Domain Decomposition Method}
\label{Sec:pd_DD}
In the primal DDM introduced in \cref{Sec:primal_DD}, 
the continuity of a solution on the subdomain interfaces is imposed directly.
Alternatively, motivated by existing DDMs in structural mechanics~\cite{FLP:2000,FR:1991}, the continuity can be enforced by the method of Lagrange multipliers, which results in a saddle point problem of the ``primal" variable $\p$ and the Lagrange multipliers $\lambda$ also known as the ``dual" variable.
We name the algorithm proposed in this section ``primal-dual DDM'' because it solves the saddle point problem of $\p$ and $\lambda$ by the primal-dual algorithm~\cite{CP:2011}.

We begin with the same domain decomposition setting as in \cref{Sec:primal_DD}.
At first, we state a proposition which suggests how to treat the continuity of the solution on the subdomain interfaces.

\begin{proposition}
\label{Prop:DD_interface}
A vector function $\q$\emph{:} $\Omega \rightarrow \R^2$ is in $\Hzdiv$ if and only if
the restriction $\q_s = \q |_{\Omega_s}$ to each subdomain $\Omega_s$ is in $H(\div; \Omega_s)$
satisfying the boundary condition $\q_s \cdot \n_s = 0$ on $\partial \Omega_s \cap \partial \Omega$
and the interface condition $\q_s \cdot \n_{st} - \q_t \cdot \n_{st} = 0$ on $\Gamma_{st}$, $s<t$.
\end{proposition}
\begin{proof}
Applying \cref{Prop:FEM_interface} to a coarse mesh $\left\{ \Omega_s \right\}_{s=1}^{\N}$ of $\Omega$ yields the conclusion.
\end{proof}

We introduce the local function space $\tY_s$, defined by
\begin{equation*}
\label{tY_s}
\tY_s = \left\{ \tilde{\q}_s \in  H(\div ; \Omega_s) : \tilde{\q}_s \cdot \n_s = 0 \textrm{ on } \partial \Omega_s \setminus \Gamma \textrm{, }
\tilde{\q}_s |_{T} \in \mathcal{RT}_0 (T) \hspace{0.2cm} \forall T \in \T_s \right\}.
\end{equation*}
The difference between $Y_s$ in~\cref{Y_s} and $\tY_s$ is that the essential boundary condition $\tilde{\q}_s \cdot \n_s = 0$ is not imposed on $\Gamma \cap \partial \Omega_s$ for $\tY_s$.
That is, $\tY_s$ has degrees of freedom on $\partial \Omega_s \cap \Gamma$ as shown in \cref{Fig:DD}(b), while $Y_s$ does not.
Let $\tilde{\I}_s$ be the set of indices of the basis functions for $\tY_s$.
Similarly to~\cref{C_s}, we define the inequality-constrained subset $\tC_s$ of $\tY_s$ by
\begin{equation*}
\label{tC_s}
\tC_s = \left\{ \tp_s \in \tY_s : |(\tp_s)_i| \leq 1 \hspace{0.2cm} \forall i \in \tilde{\I}_s \right\}.
\end{equation*}
Clearly, the projection onto $\tC_s$ is given by
\begin{equation*}
\label{proj_tC_s}
(\proj_{\tC_s} \tp_s )_i = \frac{(\tp_s)_i}{\max \left\{ 1, |(\tp_s)_i| \right\}} \hspace{0.5cm} \forall i \in \tilde{\I}_s.
\end{equation*}
Also, we denote $\tY$ by the direct sum of the local function spaces,
\begin{equation*}
\tY = \bigoplus_{s=1}^{\N} \tY_s
\end{equation*}
and we denote $\tC$ by
\begin{equation*}
\label{tC}
\tC = \bigoplus_{s=1}^{\N} \tC_s.
\end{equation*}
For $\tp = \bigoplus_{s=1}^{\N} \tp_s$, we define the energy functional $\tJ (\tp)$ on $\tY$ by
\begin{equation}
\label{tJ}
\tJ (\tp) = \sum_{s=1}^{\N} \frac{1}{2\alpha} \intOs (\div \tp_s + \alpha f)^2 \,dx.
\end{equation}
In addition, we define the operator $B$: $\tY \rightarrow \R^{|\I_{\Gamma}|}$ which measures the jump of the normal component of $\tY$ on the subdomain interfaces by
\begin{equation}
\label{B}
B\tp|_{\Gamma_{st}} = \tp_s \cdot \n_{st} - \tp_t \cdot \n_{st}, \hspace{0.5cm} s<t.
\end{equation}
Since each degree of freedom in the Raviart--Thomas elements represents the value of the normal component on the corresponding edge,
the standard matrix of $B$ consists of only $-1$'s, $0$'s, and $1$'s.
Thus, an application of $B$ can be done by a series of scalar additions/subtractions only.

By \cref{Prop:DD_interface}, there is an isomorphism between two spaces~$Y$ and~$\ker B \subset \tY$, say~$\Phi:$~$Y \rightarrow \ker B$, defined by
\begin{equation}
\label{isomorphism}
\Phi \p = \bigoplus_{s=1}^{\N} \p|_{\Omega_s}, \hspace{0.5cm} \p \in Y.
\end{equation}
By such an isomorphism, \cref{d_dual_ROF} is equivalent to
\begin{equation}
\label{pd_constrained}
\min_{\tp \in \tY} \tJ (\tp) + \chi_{\tC}(\tp) \hspace{0.5cm}
\textrm{subject to } B\tp = 0.
\end{equation}
By treating the constraint $B\tp = 0$ in~\cref{pd_constrained} by the method of Lagrange multipliers, we get the following proposition.

\begin{proposition}
\label{Prop:pd_DD_equiv}
If $\p^* \in Y$ is a solution of \cref{d_dual_ROF},
then $\Phi \p^*$ is a primal solution of the saddle point problem
\begin{equation}
\label{pd_DD}
\min_{\tp \in \tY} \max_{\lambda \in \R^{|\I_{\Gamma}|}}
\left\{ \L(\tp, \lambda ) := \tJ (\tp) + \chi_{\tC}(\tp) + \left< B\tp, \lambda \right>_{\R^{|\I_{\Gamma}|}} \right\},
\end{equation}
where~$\Phi$:~$Y \rightarrow \ker B$ was defined in~\cref{isomorphism}.
Conversely, if $\tp^* \in \ker B \subset \tY$ is a primal solution of \cref{pd_DD},
then $\Phi^{-1} \tp^*$ is a solution of \cref{d_dual_ROF}.
\end{proposition}

Since the functional $\tJ (\tp)$ in \cref{tJ} is convex but not uniformly convex, the $O(1/n)$-primal-dual algorithm can be utilized to solve~\cref{pd_DD}~\cite{CP:2011}.
To estimate a valid range of parameters for the primal-dual algorithm, \cref{Lem:pd_DD_B_norm} gives a norm bound of the operator $B:\tY \rightarrow \R^{|\I_{\Gamma}|}$.

\begin{lemma}
\label{Lem:pd_DD_B_norm}
The operator norm of $B$\emph{:} $\tY \rightarrow \R^{|\I_{\Gamma}|}$ defined in~\cref{B} has a bound such that $\| B \|_{\tY \rightarrow \R^{|\I_{\Gamma}|}}^2 \leq 2$.
\end{lemma}
\begin{proof}
Fix $\tp = \bigoplus_{s=1}^{\N} \tp_s \in \tY$.
Let $(B\tp)_i$ be a degree of freedom of $B\tp$ on $\Gamma_{st}$ for some $s<t$,
and let $(\tp_s)_i$, $(\tp_t)_i$ be degrees of freedom of $\tp_s$, $\tp_t$ adjacent to $(B\tp)_i$, respectively.
Then it satisfies that
$$ (B\tp)_i = (\tp_s)_i - (\tp_t)_i .$$
By applying the Cauchy--Schwarz inequality, we get
$$ (B\tp)_i^2 \leq 2((\tp_s)_i^2 + (\tp_t)_i^2).$$
Summation over every $i$ and $s<t$ yields $\| B\tp \|_{\R^{|\I_{\Gamma}|}}^2 \leq 2 \| \tp \|_{\tY}^2$.
\end{proof}

Thanks to \cref{Lem:pd_DD_B_norm}, the primal-dual algorithm for~\cref{pd_DD} is given in \cref{Alg:pd_DD}.
We notice that the primal-dual algorithm was used for DDMs in~\cite{DCT:2016}.

\begin{algorithm}[]
\caption{Primal-dual DDM}
\begin{algorithmic}[]
\label{Alg:pd_DD}
\STATE Choose $L \geq 2$, $\tau, \sigma > 0$ with $\tau \sigma = \frac{1}{L}$.
Let $\tp^{(0)} = \0$ and $\lambda^{(0)} = 0$.
\FOR{$n=0,1,2,...$}
\STATE $\displaystyle \lambda^{(n+1)} = \lambda^{(n)} + \sigma B (2\tp^{(n)} - \tp^{(n-1)} )$
\STATE $\displaystyle \tp^{(n+1)} \in \argmin_{\tp \in \tY} \left\{  \tJ (\tp) + \chi_{\tC}(\tp) + \frac{1}{2\tau} \intO (\tp - \hat{\p})^2 \,dx \right\}$,\\
\quad where $\displaystyle \hat{\p} = \tp^{(n)} - \tau B^* \lambda^{(n+1)}$
\ENDFOR
\end{algorithmic}
\end{algorithm}

We note that the primal problem for $\tp^{(n+1)}$ in \cref{Alg:pd_DD} can be solved independently in each subdomain.
Indeed, $\tp^{(n+1)}$ can be obtained as the direct sum of $\tp_s^{(n+1)}$'s, where $\tp_s^{(n+1)}$ is a solution of
\begin{equation}
\label{pd_DD_local}
\min_{\tp_s \in \tY_s} \left\{ \frac{1}{2\alpha} \intOs (\div \tp_s + \alpha f)^2 \,dx + \chi_{\tC_s}(\tp_s) + \frac{1}{2\tau} \intOs (\tp_s - \hat{\p}_s)^2 \,dx \right\},
\end{equation}
where $\hat{\p}_s = \tilde{\p}_s^{(n)} - \tau B^* \lambda^{(n+1)} |_{\Omega_s}$.
Now, we state the convergence analysis for \cref{Alg:pd_DD}.
See Theorem~5.1 of~\cite{CP:2016} for details.

\begin{theorem}
\label{Thm:pd_DD}
Let $\left\{ \tp^{(n)}, \lambda^{(n)} \right\}$ be the sequence generated by \cref{Alg:pd_DD}.
Then, it converges to a saddle point of~\cref{pd_DD} and satisfies that
$$
\L \left( \frac{1}{n}\sum_{k=1}^{n}\tp^{(k)}, \lambda \right) - 
 \L \left( \tp, \frac{1}{n}\sum_{k=1}^{n}\lambda^{(k)} \right)
 \leq \frac{1}{n} \left( \frac{1}{\tau} \| \tp - \tp^{(0)} \|_{2, \tY}^2 + \frac{1}{\sigma} \| \lambda - \lambda^{(0)} \|_{2, \mathbb{R}^{|\I_{\Gamma}|}}^2 \right)
$$
for any $\tp \in Y$ and $\lambda \in \mathbb{R}^{|\I_{\Gamma}|}$.
\end{theorem}

Even though the convergence rate in \cref{Thm:pd_DD} is the same as the existing methods (see, e.g.,~\cite{CTWY:2015}),
the proposed primal-dual DDM has an advantage for the convergence rate of local problems compared to the existing ones.
With the help of a $\frac{1}{\tau}$-uniformly convex term
$$\frac{1}{2\tau} \intOs (\tp_s - \hat{\p}_s)^2 \,dx$$
in~\cref{pd_DD_local}, linearly convergent algorithms such as~\cite[Algorithm~3]{CP:2011} and~\cite[Algorithm~5]{CP:2016} can be adopted, while the known optimal convergence rate of the existing methods for the ROF model is only $O(1/n^2)$, which is far slower than linear convergence.
The following is the linearly convergent primal-dual algorithm~\cite[Algorithm~3]{CP:2011} applied to~\cref{pd_DD_local}.

\begin{algorithm}[]
\renewcommand{\thealgorithm}{}
\caption{Linearly convergent local solver for \cref{Alg:pd_DD}}
\begin{algorithmic}[]
\label{Alg:pd_DD_local}
\STATE Choose $L \geq 8$, $\gamma \leq \alpha$, and $\delta \leq \frac{1}{\tau}$.
\STATE Set $\mu = \frac{2\sqrt{\gamma \delta}}{L}$, $\tau_0 = \frac{\mu}{2\gamma}$, $\sigma_0 = \frac{\mu}{2\delta}$, and $\theta_0 \in \left[ \frac{1}{1+\mu}, 1\right]$.
Let $\bar{u}_s^{(0)} = u_s^{(0)} = 0$ and $\tp_s^{(0)} = \0$.
\FOR{$n=0,1,2,...$}
\STATE $\displaystyle \tp_s^{(n+1)} = \proj_{\tC_s} \left( \frac{\tau (\tp_s^{(n)} - \sigma_0 \div^* \bar{u}_s^{(n)}) + \sigma_0 \hat{\p}_s}{\tau + \sigma_0}\right)$
\STATE $\displaystyle u_s^{(n+1)} = \frac{(u_s^{(n)} + \tau_0 \div \p_s^{(n+1)}) + \tau_0 \alpha f}{1 + \tau_0 \alpha}$
\STATE $\bar{u}_s^{(n+1)} = u_s^{(n+1)} + \theta_0 (u_s^{(n+1)} - u_s^{(n)})$
\ENDFOR
\end{algorithmic}
\end{algorithm}

\section{Numerical Results}
\label{Sec:numerical}
In this section, numerical results of the algorithms introduced in previous sections are presented.
All the algorithms were implemented in MATLAB~R2018a,
and all the computations were performed on a desktop equipped with Intel Core i5-8600K CPU (3.60GHz), 16GB memory, and the OS Windows 10 Pro 64-bit.
Two test images ``Peppers $512\times512$" and ``Boat $2048 \times 3072$," shown in \cref{Fig:test_images}, were used in the numerical experiments.
We introduced noise to each image using Gaussian additive noise with mean $0$ and variance $0.05$.
As a measurement of the quality of denoising, the peak-signal-to-noise ratio (PSNR) defined by
\begin{equation*}
\mathrm{PSNR} = 10 \log_{10} \left( \frac{\mathrm{MAX}^2 \cdot |\Omega|}{\| u-f_{\mathrm{orig}}\|_{X}^2} \right),
\end{equation*}
where $\mathrm{MAX}$ is the maximum possible pixel value of the image ($\mathrm{MAX} = 1$ in our experiments), $f_{\mathrm{orig}}$ is the original clean image and $u$ is a denoised image, is calculated for each output of the experiment.
We set~$\alpha = 10$ heuristically in~\cref{ROF}.

\begin{figure}[]
\centering
\subfloat[][Peppers $512 \times 512$]{ \includegraphics[height=3.8cm]{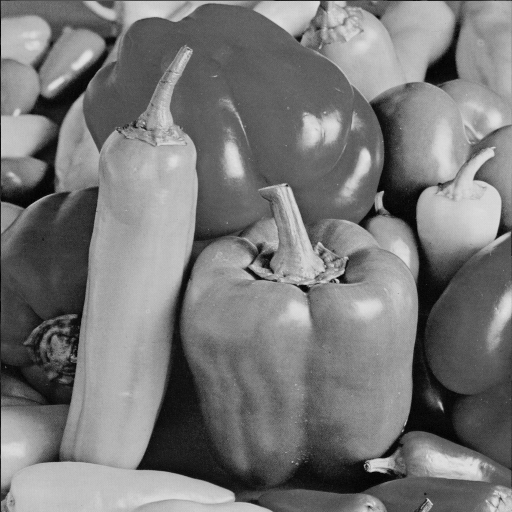} }
\hspace{1.2cm}
\subfloat[][Boat $2048 \times 3072$]{ \includegraphics[height=3.8cm]{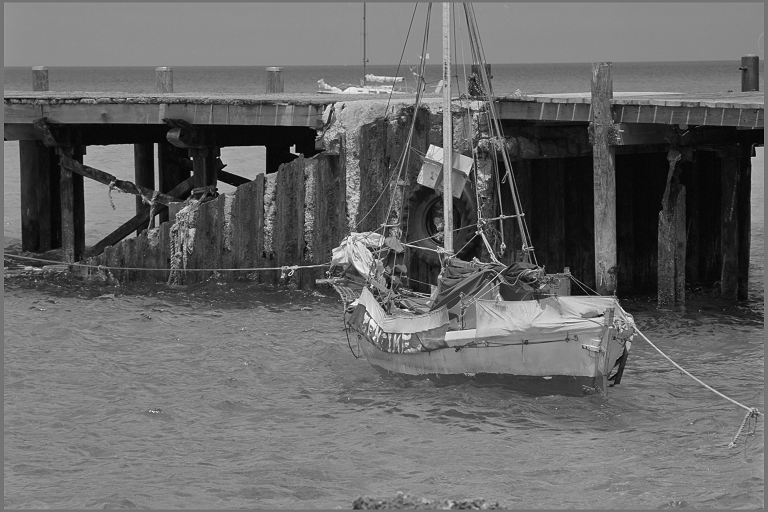} }
\caption{Test images for the numerical experiments}
\label{Fig:test_images}
\end{figure}

First, we compare the proposed methods with other existing DDMs for the ROF model.
Thanks to~\cref{Thm:equiv}, direct comparisons with existing methods based on the finite difference discretization are available in the aspect of the primal energy functional defined as
\begin{equation}
\label{primal_energy}
\E (u) = \frac{\alpha}{2} \| u - f \|_2^2 + \| | Du|_1 \|_1.
\end{equation}
The following algorithms are used for our numerical experiments:
\begin{itemize}
\item ALG1: Primal DDM described in \cref{Alg:primal_DD}, $L=4$.
\item ALG2: Primal-dual DDM described in \cref{Alg:pd_DD}, $L=2$, $\sigma=0.02$, $\sigma \tau = 1/L$.
\item HL--RJ: Relaxed block Jacobi~(parallel) method proposed by Hinterm{\"u}ller and Langer~\cite{HL:2015}, relaxation parameter:~$1/3$ (see Remark~3.3 of~\cite{LN:2017}).
\item HL--GS: Block Gauss--Seidel~(successive) method proposed by Hinterm{\"u}ller and Langer~\cite{HL:2015}.
\item LN--RJ: Relaxed block Jacobi method proposed by Lee and Nam~\cite{LN:2017}, relaxation parameter:~$1/3$.
\item LN--GS: Block Gauss--Seidel method proposed by Lee and Nam~\cite{LN:2017}.
\end{itemize}
The number of subdomains~$\N$ is fixed at~$4\times4$.
Local problems are solved by the~$O(1/n^2)$ convergent primal-dual algorithm~\cite[Algorithm~2]{CP:2011} with the parameters $L=8$, $\gamma = 0.125\alpha$, $\tau_0 = 0.01$, and $\sigma_0 \tau_0 = 1/L$ for all algorithms stated above but ALG2.
For ALG2, the linearly convergent primal-dual algorithm~\cite[Algorithm~3]{CP:2011} with the parameters~$L=8$, $\gamma = 0.5\alpha$, and $\delta = 1/\tau$ are used.
Local problems are solved by the following stop criterion:
\begin{equation*}
\frac{\| \p_s^{(n+1)} - \p_s^{(n)} \|_2}{\| \p_s^{(n+1)}\|_2} < 10^{-8}.
\end{equation*}

To evaluate the performances of DDMs based on iterations of the dual variables~$\left\{\p^{(n)} \right\}$ in terms of the primal energy~\cref{primal_energy}, we have to define the primal iterates~$\left\{u^{(n)} \right\}$ appropriately.
For HL--RJ and HL--GS, we define $u^{(n)}$ as
\begin{equation*}
u^{(n)} = f + \frac{1}{\alpha} \div \p^{(n)}.
\end{equation*}
Also, for ALG1 and ALG2, $u^{(n)}$ is defined as
\begin{equation}
\label{primal_DD_u}
u^{(n)} = f + \frac{1}{\alpha} \div (\Hq^{(n)} \oplus \q_{\Gamma}^{(n)})
\end{equation}
and
\begin{equation}
\label{pd_DD_u}
u^{(n)} = f + \frac{1}{\alpha} \bigoplus_{s=1}^{\N} \div \tp_s^{(n)} ,
\end{equation}
respectively.
Meanwhile, we compute the minimum value of the primal energy~$\E (u^*)$ approximately by 10,000 iterations of the~$O(1/n^2)$ convergent primal-dual algorithm applied to the full dimension problem~\cref{d_dual_ROF}.

\begin{figure}[]
\centering
\subfloat[][Peppers $512 \times 512$]{ \includegraphics[height=4.9cm]{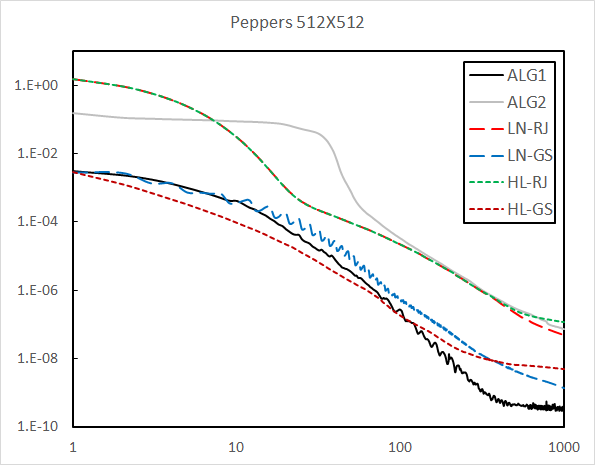} }
\subfloat[][Boat $2048 \times 3072$]{ \includegraphics[height=4.9cm]{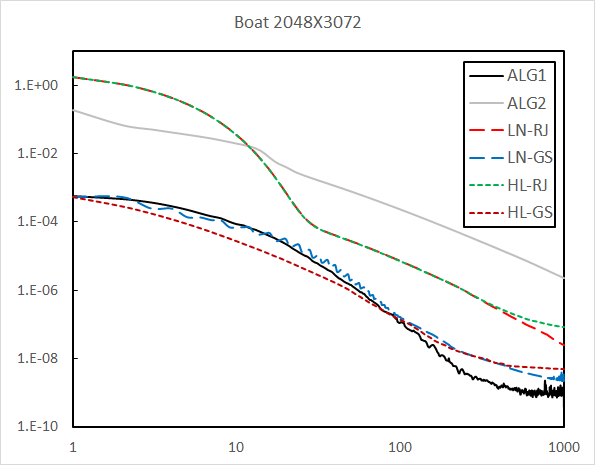} }
\caption{Decay of the values of~$\frac{\E (u^{(n)}) - \E(u^*) }{\E (u^*)}$ in various DDMs for the ROF model}
\label{Fig:comp}
\end{figure}

\Cref{Fig:comp} shows the decay of the relative primal energy functional~$\frac{\E (u^{(n)}) - \E(u^*) }{\E (u^*)}$ during 1,000 outer iterations for various DDMs.
It can be observed that the primal energy of ALG1 decreases as fast as the block Gauss--Seidel methods.
ALG1 has an advantage compared to the block Gauss--Seidel methods in the aspect of parallel computation; all local problems of ALG1 can be solved in parallel while only local problems of the same color can be solved in parallel for the block Gauss--Seidel methods.
In~\cref{Fig:comp}, there are oscillations of the primal energy of ALG1 when the value of~$\frac{\E (u^{(n)}) - \E(u^*) }{\E (u^*)}$ is close to~$10^{-10}$.
This is because local problems are solved inexactly by iterative methods.

\begin{table}[]
\centering
\begin{tabular}{| c | c c c c c c |} \hline
Test image & ALG1 & ALG2 & HL--RJ & HL--GS & LN--RJ & LN--GS \\
\hline
Peppers
& 2923 & 271 &  2925 & 2928 & 2925 & 2929 \\ \cline{2-6}
\hline
Boat
& 8613 & 263 & 8613  & 8613 & 8613 & 8614 \\ \cline{2-6}
\hline
\end{tabular}
\caption{Maximum numbers of inner iterations in various DDMs for the ROF model}
\label{Table:max_inner_iter}
\end{table}

Even though the primal energy of ALG2 does not decrease faster than the existing methods, it has its own advantage in that local problems can be solved much faster.
\Cref{Table:max_inner_iter} shows the maximum numbers of inner iterations during~1,000 outer iterations for various DDMs.
ALG1 shows similar behavior on inner iterations compared to the existing DDMs.
On the other hand, as we explained in~\cref{Sec:pd_DD}, ALG2 can adopt linearly convergent algorithms as local solvers, while the other algorithms cannot.
Thus, the maximum number of inner iterations of ALG2 is much less than the other ones.
This phenomenon makes ALG2 practically efficient.
For example, in the case of the test image ``Boat $2048\times3072$,'' a single outer iteration of ALG2 is approximately 32 times faster than the other methods.

Next, we present numerical results for the proposed methods, which emphasize their efficiency as parallel solvers.
To evaluate the parallel efficiency, the \textit{virtual wall-clock time} is measured, which assumes that the algorithms run in parallel in each subdomain.
That is, it ignores the communication time among processors.

We first present the numerical results for \cref{Alg:primal_DD}.
We set the parameter $L = 4$.
We note that, in the viewpoint of image restoration, the stop criteria for the proposed methods need not to be too strict.
We use the following stop criterion:
\begin{equation}
\label{stop_outer}
\left| \frac{\E(u^{(n+1)}) - \E(u^{(n)})}{\E(u^{(n+1)})} \right| < 10^{-3} ,
\end{equation}
where $u^{(n)}$ was defined in~\cref{primal_DD_u}.
Local problems are solved by the $O(1/n^2)$ convergent primal-dual algorithm with the parameters $L=8$, $\gamma = 0.125\alpha$, $\tau_0 = 0.01$, and $\sigma_0 \tau_0 = 1/L$ and the stop criterion
\begin{equation}
\label{stop_inner}
\frac{\| \p_{s}^{(n+1)} - \p_{s}^{(n)} \|_{Y_{s}}}{\| \p_{s}^{(n+1)} \|_{Y_{s}}} < 10^{-5} .
\end{equation}

\begin{table}[]
\centering
\begin{tabular}{| c | c | c c c c |} \hline
Test image & $\N$ & PSNR & iter & \begin{tabular}{c}max\\inner iter\end{tabular} & \begin{tabular}{c}Virtual\\wall-clock\\time (sec)\end{tabular} \\
\hline
\multirow{4}{*}{\shortstack{\begin{phantom}1\end{phantom} \\ \begin{phantom}2\end{phantom} \\ Peppers \\ $512 \times 512$}}
& 1 & 24.41 & - & 526 & 4.90 \\ \cline{2-6}
& $2 \times 2$ & 24.41 & 2 & 532 & 0.77 \\
& $4 \times 4$ & 24.41 & 2 & 584 & 0.26 \\
& $8 \times 8$ & 24.41 & 5 & 590 & 0.22 \\
& $16 \times 16$ & 24.41 & 7 & 573 & 0.14 \\
\hline
\multirow{4}{*}{\shortstack{\begin{phantom}1\end{phantom} \\ \begin{phantom}2\end{phantom} \\ Boat \\ $2048 \times 3072$}}
& 1 & 24.75 & - & 995 & 273.48 \\ \cline{2-6}
& $2 \times 2$ & 24.75 & 2 & 1145 & 91.72 \\
& $4 \times 4$ & 24.75 & 2 & 1408 & 21.03 \\
& $8 \times 8$ & 24.75 & 2 & 1415 & 3.42 \\
& $16 \times 16$ & 24.75 & 2 & 1492 & 1.31 \\
\hline
\end{tabular}
\caption{Performance of the primal DDM \cref{Alg:primal_DD}}
\label{Table:primal_DD}
\end{table}

\cref{Table:primal_DD} shows the performance of \cref{Alg:primal_DD}.
For the single subdomain case, the $O(1/n^2)$ convergent primal-dual algorithm is used.
The PSNRs of the resulting denoised images do not differ from the single subdomain case.
Thus, we can conclude that the results of \cref{Alg:primal_DD} agree with the single subdomain case, as proven in \cref{Prop:primal_DD_equiv}.
With sufficiently many subdomains, the virtual wall-clock time is much less than the wall-clock time of the single subdomain case.
It shows the worth of \cref{Alg:primal_DD} as a parallel algorithm.

Next, we consider the primal-dual DDM.
For \cref{Alg:pd_DD}, we set the parameters $L=2$, $\sigma = 0.02$, and $\sigma\tau = 1/L$.
We use the same stop criterion~\cref{stop_outer} for the outer iterations as in~\cref{Alg:primal_DD} with~$u^{(n)}$ defined in~\cref{pd_DD_u}.
For the local solver, the parameters $L=8$, $\gamma = 0.5\alpha$, and $\delta = 1/\tau$ are used.
The stop criterion~\cref{stop_inner} for local problems is used for~$\tilde{\p}_s^{(n)}$.

\begin{table}[]
\centering
\begin{tabular}{| c | c | c c c c |} \hline
Test image & $\N$ & PSNR & iter & \begin{tabular}{c}max\\inner iter\end{tabular} & \begin{tabular}{c}Virtual\\wall-clock\\time (sec)\end{tabular} \\
\hline
\multirow{4}{*}{\shortstack{\begin{phantom}1\end{phantom} \\ \begin{phantom}2\end{phantom} \\ Peppers \\ $512 \times 512$}}
& 1 & 24.41 & - & 526 & 4.90 \\ \cline{2-6}
& $2 \times 2$ & 24.41 & 22 & 144 & 2.09 \\
& $4 \times 4$ & 24.41 & 24 & 147 & 0.66 \\
& $8 \times 8$ & 24.41 & 26 & 150 & 0.28 \\
& $16 \times 16$ & 24.41 & 30 & 154 & 0.19 \\
\hline
\multirow{4}{*}{\shortstack{\begin{phantom}1\end{phantom} \\ \begin{phantom}2\end{phantom} \\ Boat \\ $2048 \times 3072$}}
& 1 & 24.75 & - & 995 & 273.48 \\ \cline{2-6}
& $2 \times 2$ & 24.75 & 12 & 138 & 95.84 \\
& $4 \times 4$ & 24.75 & 18 & 140 & 24.59 \\
& $8 \times 8$ & 24.75 & 20 & 144 & 3.38 \\
& $16 \times 16$ & 24.75 & 24 & 146 & 1.74 \\
\hline
\end{tabular}
\caption{Performance of the primal-dual DDM \cref{Alg:pd_DD}}
\label{Table:pd_DD}
\end{table}

As \cref{Table:pd_DD} shows, the solution of \cref{Alg:pd_DD} is consistent with the single subdomain case regardless of the number of subdomains.
Since the local solver has the linear convergence rate, which is much faster than the standard algorithms for the ROF model,
we can observe that the maximum number of inner iterations of \cref{Alg:pd_DD} is smaller than that of \cref{Alg:primal_DD} in all cases.
For example, in the experiments with the test image ``Boat $2048 \times 3072$,'' local problems of \cref{Alg:pd_DD} are solved approximately 10 times faster than those of \cref{Alg:primal_DD}.
Consequently, even though the convergence rate of \cref{Alg:pd_DD} is only~$O(1/n)$, the virtual wall-clock time of \cref{Alg:pd_DD} is as small as that of~\cref{Alg:primal_DD} in the case of sufficiently many subdomains.

\begin{figure}[]
\centering
\subfloat[][Noisy ``Peppers $512\times512$'' \\ \centering(PSNR: 19.11)]{ \includegraphics[width=3.8cm]{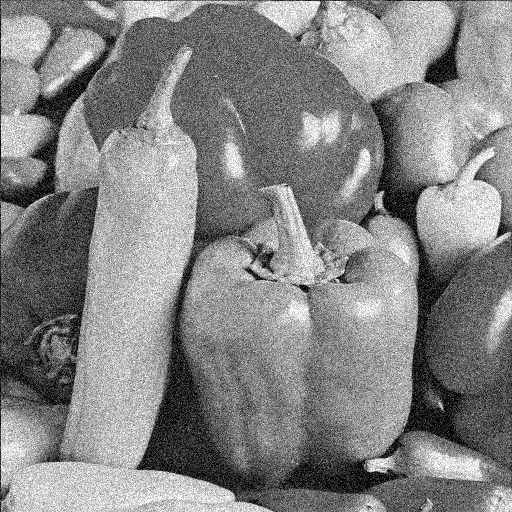} }
\subfloat[][Primal DDM, $\N = 16\times16$ \\ \centering(PSNR: 24.41)]{ \includegraphics[width=3.8cm]{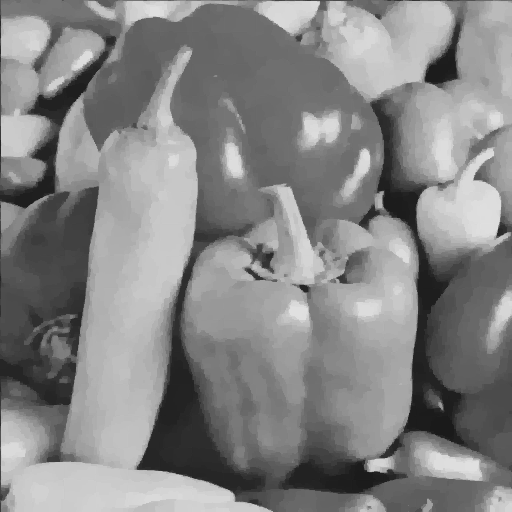} }
\subfloat[][Primal-dual DDM,\\ \centering$\N = 16\times16$ (PSNR: 24.41)]{ \includegraphics[width=3.8cm]{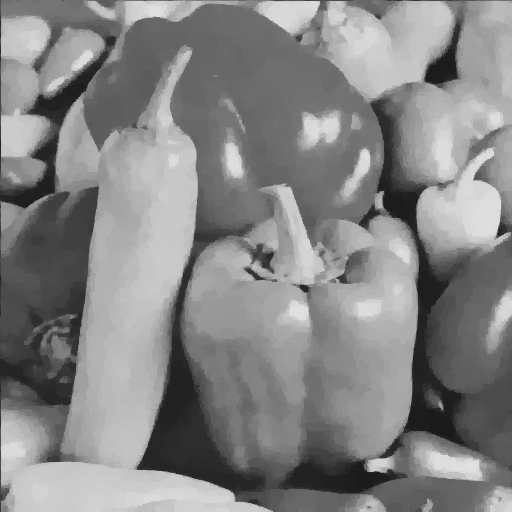} } 

\subfloat[][Noisy ``Boat $2048\times3072$'' \\ \centering(PSNR: 19.10)]{ \includegraphics[width=3.8cm]{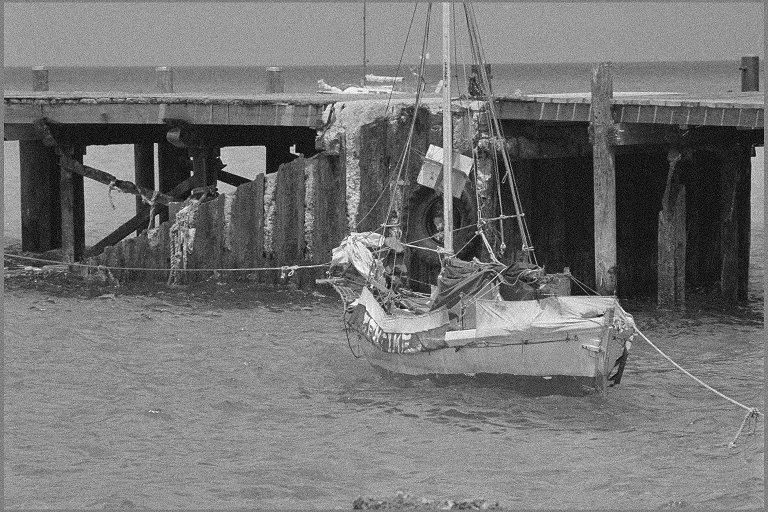} }
\subfloat[][Primal DDM, $\N = 16\times16$ \\ \centering(PSNR: 24.75)]{ \includegraphics[width=3.8cm]{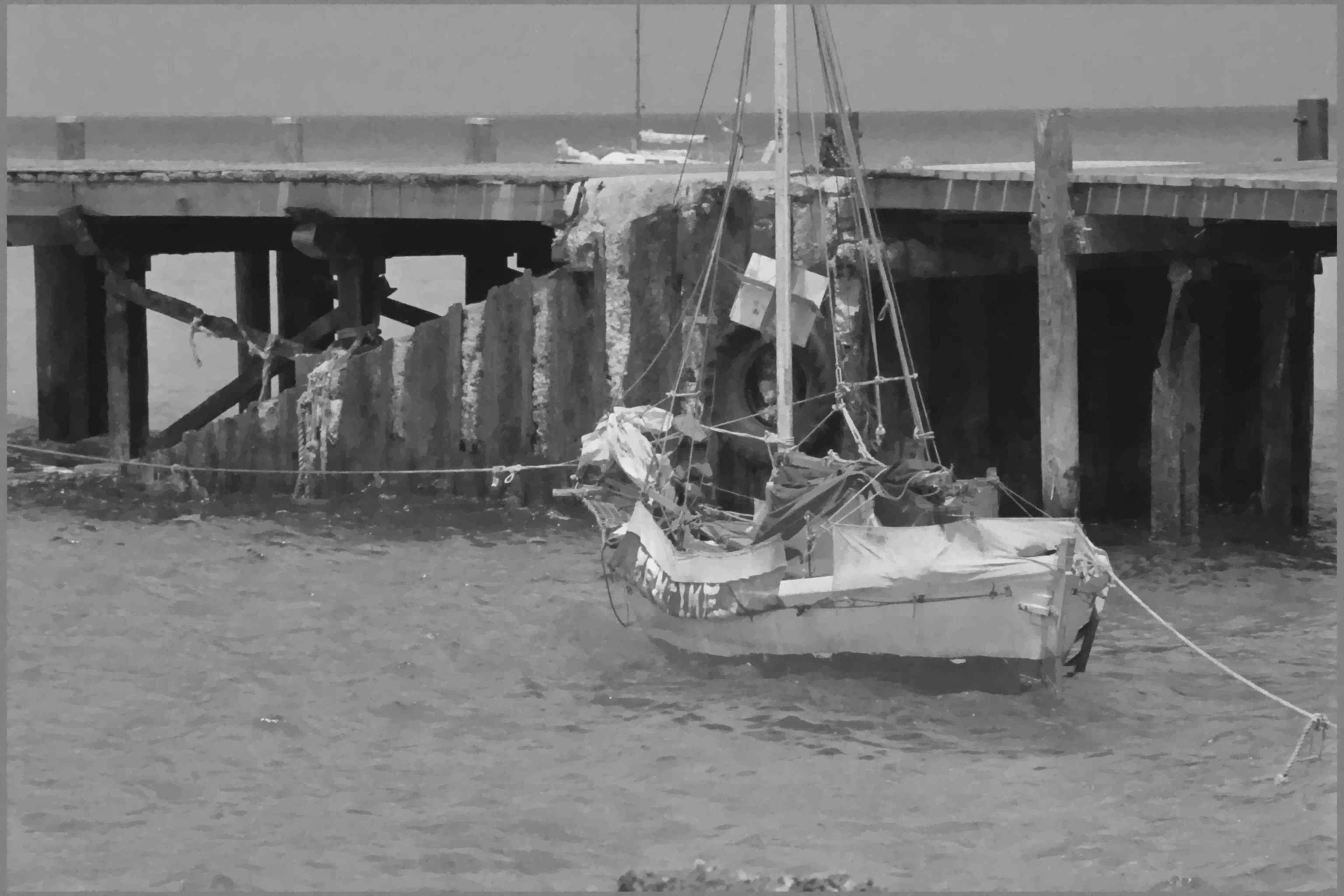} }
\subfloat[][Primal-dual DDM,\\ \centering$\N = 16\times16$ (PSNR: 24.75)]{ \includegraphics[width=3.8cm]{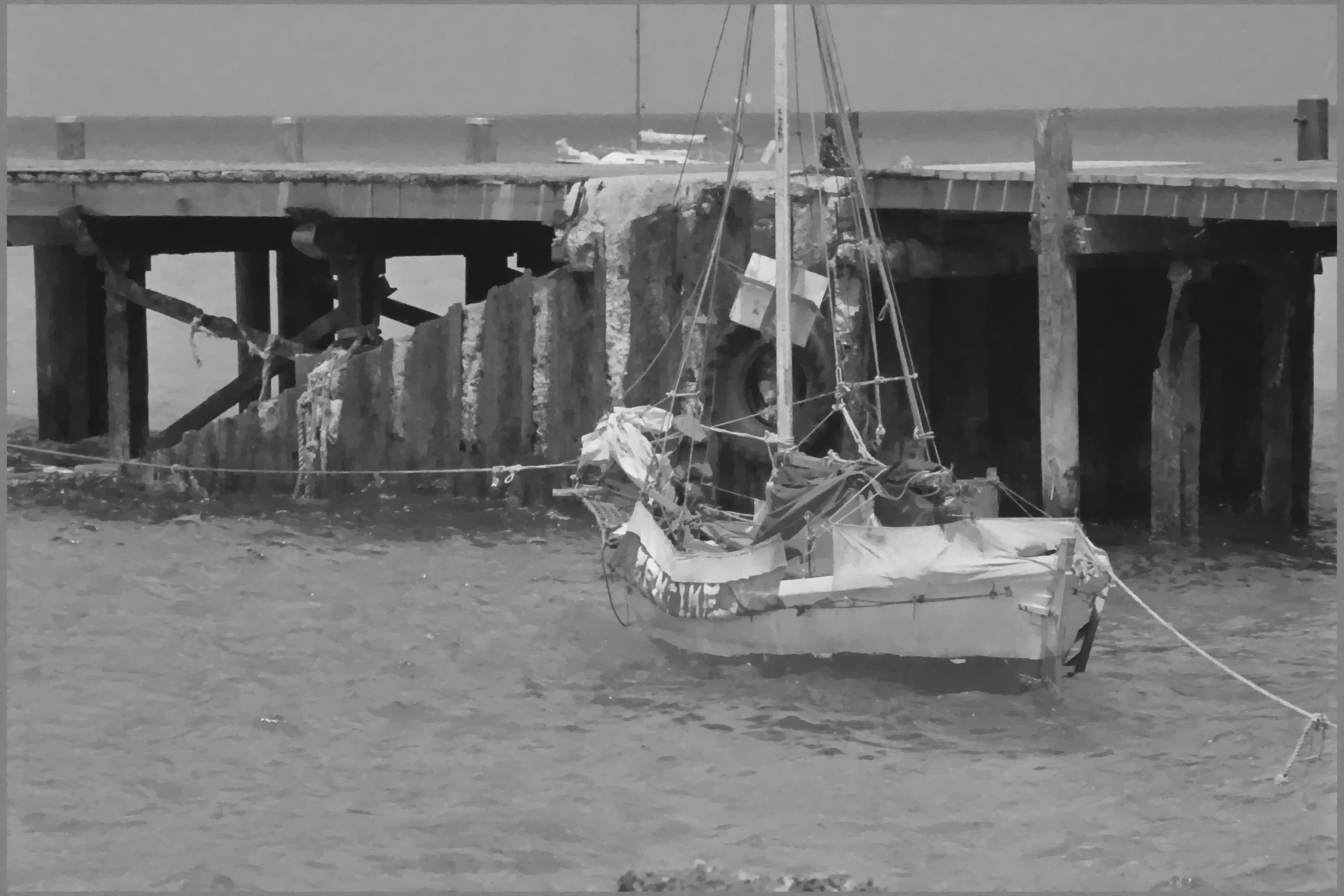} } 

\caption{Results of \cref{Alg:primal_DD,Alg:pd_DD} for test images}
\label{Fig:denoised_images}
\end{figure}

Finally, we display the resulting denoised images by the proposed DDMs in \cref{Fig:denoised_images}.
We only provide the images for the case $\N = 16 \times 16$ since all the resulting images are visually the same regardless of the number of subdomains.
One can observe that there are no artificialities at all on the subdomain interfaces even in the case of quite large number of subdomains.

\section{Conclusion}
\label{Sec:conclusion}
In this paper, we proposed an alternative discretization~\cref{d_dual_ROF} for the dual ROF model using a conforming Raviart--Thomas basis.
We mentioned that the proposed discretization naturally satisfies the splitting property~\cref{splitting} of the energy functional.
Thanks to the splitting property, we proposed two DDMs for the dual ROF model: the primal one and the primal-dual one.
We showed that the proposed primal DDM has a $O(1/n^2)$ convergence rate, which is the best among the existing DDMs.
Also, we showed that the local problems in the proposed primal-dual DDM can be solved at a linear convergence rate by using the accelerated primal-dual algorithm.
Numerical results demonstrate the superiority of the proposed DDMs.

We conclude the paper with a remark on the primal-dual DDM.
 Since we did not use any regularity of the dual ROF energy functional to prove convergence of the primal-dual DDM, we expect that the primal-dual DDM can be generalized to more advanced imaging problems with total variation, for example, total variation minimization with~$L^1$-fidelity term~\cite{CE:2005}.
 
\bibliographystyle{siamplain}
\bibliography{references}
\end{document}